\documentclass{amsart}
\usepackage{math,amsrefs,bbold,mathrsfs,amscd,lineno,algorithm2e}
\linenumbers
\newcommand\Fp{{\mathbb F_p}}
\newcommand\Ft{{\mathbb F_2}}
\newcommand\LL{{\mathscr L}}
\newcommand\LLgs{{\LL_{\text{GS}}}}
\newcommand\LLg{{\LL_{\text{G}}}}
\newcommand\LLtg{{{}_2\LLg}}
\renewcommand\AA{{\mathscr A}}
\newcommand\II{{\mathscr I}}
\newcommand\KK{{\mathscr K}}
\newcommand\W{{\mathscr W}}
\newcommand\Der{{\operatorname{\mathfrak{Der}}}}
\newcommand\U{{\mathbb U}}
\newcommand\Hom{{\operatorname{Hom}}}
\newcommand\Sym{{\operatorname{Sym}}}
\newcommand\Mat{{\operatorname{Mat}}}
\newcommand\Hdim{{\operatorname{Hdim}}}
\newcommand\GKdim{{\operatorname{GKdim}}}

\begin{document}
\title{Self-similar Lie algebras}
\author{Laurent Bartholdi}
\address{Laurent Bartholdi, Mathematisches Institut, Georg-August Universit\"at zu G\"ottingen}
\email{laurent.bartholdi@gmail.com}
\date{March 1, 2010}
\begin{abstract}
  We give a general definition of self-similar Lie algebras, and show
  that important examples of Lie algebras fall into that class. We
  give sufficient conditions for a self-similar Lie algebra to be nil,
  and prove in this manner that the self-similar algebras associated
  with Grigorchuk's and Gupta-Sidki's torsion groups are nil as well
  as self-similar. We derive the same results for a class of examples
  constructed by Petrogradsky, Shestakov and Zelmanov.
\end{abstract}
\maketitle

%%%%%%%%%%%%%%%%%%%%%%%%%%%%%%%%%%%%%%%%%%%%%%%%%%%%%%%%%%%%%%%%
\section{Introduction}
Since its origins, mankind has been divided into hunters and
gatherers. This paper is resolutely of the latter kind, and brings
together Caranti \emph{et al.}'s Lie algebras of maximal
class~\cites{caranti-m-n:glamc1,caranti-n:glamc2,caranti-vl:glamc4,jurman:glamc3,caranti-vl:glamc5},
the self-similar lie algebras associated with self-similar groups
from~\cite{bartholdi:lcs}, the self-similar associative algebras
from~\cite{bartholdi:branchalgebras}, and Petrogradsky \emph{et al.}'s
nil lie
algebras~\cites{petrogradsky:sila,petrogradsky-s:sila,%petrogradsky-s-z:nillie,
  shestakov-z:nillie}. Contrary to tradition~\cite{?:bible}*{Gen~4.8},
we do not proclaim superiority of gatherers; yet we reprove, in what
seems a more natural language, the main results of these last
papers. In particular, we extend their criteria for growth
(Propositions~\ref{prop:GKupper} and~\ref{prop:GKlower}) and nillity
(Corollary~\ref{cor:nil}).

The fundamental notion is that of a \emph{self-similar algebra},
see~Definition~\ref{defn:ssla}. If $\LL$ be a Lie algebra and $X$ be a
commutative algebra with Lie algebra of derivations $\Der X$, their
\emph{wreath product} $\LL\wr\Der X$ is the Lie algebra $\LL\otimes
X\rtimes\Der X$. A \emph{self-similar Lie algebra} is then a Lie
algebra endowed with a map $\LL\to\LL\rtimes\Der X$.

We give in~\S\S\ref{ss:ssla}--\ref{ss:wbr} a sufficient condition on a
self-similar Lie algebra $\LL$ to be nil, namely for $\ad(x)$ to be a
nilpotent endomorphism of $\LL$ for all $x\in\LL$. Our condition is
then applied in~\S\ref{ss:examples} to all the known examples, and
provides uniform proofs of their nillity.

In particular, we show that the Lie algebras
in~\cites{petrogradsky:sila,shestakov-z:nillie} are self-similar, a
fact hinted at, but never explicitly stated or used in these papers.

Golod constructed in~\cites{golod:nil,golod-s:orig} infinite
dimensional, finitely generated nil associative algebras. These
algebras have exponential growth --- indeed, this is how they are
proved to be infinite dimensional --- and quite intractable. In
searching for more examples, Small asked whether there existed such
examples with finite Gelfand-Kirillov dimension (i.e., roughly
speaking, polynomial growth). That question was answered positively by
Lenagan and Smoktunowicz in~\cite{lenagan-s:nillie}.

The question may also be asked for Lie algebras; one then has the
concrete constructions described in~\S\ref{ss:examples}. Alas, none of
their enveloping algebras seem to be nil, though they may be written
as the sum of two nil subalgebras (Proposition~\ref{prop:sum}); in
particular, all their homogeneous elements are nil.

In~\S\ref{ss:assoc} we construct a natural self-similar associative
algebra from a self-similar Lie algebra. In~\S\ref{ss:ssg} we
construct a natural self-similar Lie algebra from a self-similar group
acting `cyclically' on an alphabet of prime order. We relate in this
manner the examples from~\S\ref{ss:examples} to the well-studied
Grigorchuk~\cite{grigorchuk:burnside} and
Gupta-Sidki~\cite{gupta-s:burnside} groups:
\begin{thm}
  The Lie algebra associated with the Grigorchuk group, after
  quotienting by its centre, is self-similar, nil, and of maximal
  class. Its description as a self-similar algebra is given
  in~\S\ref{ss:grla}.

  The Lie algebra associated with the Gupta-Sidki group is
  self-similar, nil, and of finite Gelfand-Kirillov dimension. Its
  description as a self-similar algebra is given in~\S\ref{ss:gsla}.
\end{thm}

\noindent We might dare the following
\begin{conj}
  If $G$ is a torsion group, then its associated Lie algebra is nil.
\end{conj}

\subsection{Preliminaries}\label{ss:notation}
All our algebras (Lie or associative) are over a commutative domain
$\Bbbk$. If $\Bbbk$ is of positive characteristic $p$, a Lie algebra
$\LL$ over $\Bbbk$ may be \emph{restricted}, in the sense that it
admits a semilinear map $x\mapsto x^p$ following the usual axioms of
raising-to-the-$p$th-power, e.g.\ $(\xi a)^2=\xi^2a^2$ and
$(a+b)^2=a^2+b^2+[a,b]$ if $p=2$. Note that, if $\LL$ is centreless,
then the $p$-mapping is unique if it exists. Unless otherwise stated
(see e.g.~\S\ref{ss:gla}), we will assume our algebras are
restricted. We will also, by convention, say that algebras in
characteristic $0$ are restricted. See~\cite{jacobson:restr}
or~\cite{jacobson:lie}*{Chapter~V} for an introduction to restricted
algebras.

The tensor algebra over a vector space $V$ is written $T(V)$. It is
the free associative algebra $\bigoplus_{n\ge0}V^{\otimes n}$. We
recall that every Lie algebra has a \emph{universal enveloping
  algebra}, unique up to isomorphism. In characteristic $0$, or for
unrestricted Lie algebras, it is the associative algebra
\[\U(\LL)=T(\LL)/\langle a\otimes b-b\otimes a-[a,b]\text{ for all }a,b\in\LL\rangle.\]
If however $\LL$ is restricted, it is the associative algebra
\[\U(\LL)=T(\LL)/\langle a\otimes b-b\otimes a-[a,b]\text{ for all
}a,b\in\LL,a^{\otimes p}-a^p\text{ for all }a\in\LL\rangle.\]
No confusion should arise, because it is always the latter algebra
that is meant in this text if $\Bbbk$ has positive characteristic.

Wreath products of Lie algebras have appeared in various places in the
literature~\cites{jurman:glamc3,caranti-m-n:glamc1,netreba-s:wreath,dipietro:phd,shmelkin:wreath,sullivan:wreath}. We
use them as a fundamental tool in describing and constructing Lie
algebras; though they are essentially (at least, in the case of a
wreath product with the trivial Lie algebra $\Bbbk$) equivalent to the
inflation/deflation procedures of~\cite{caranti-m-n:glamc1}.

%%%%%%%%%%%%%%%%%%%%%%%%%%%%%%%%%%%%%%%%%%%%%%%%%%%%%%%%%%%%%%%%
\section{Self-similar Lie algebras}\label{ss:ssla}
Just as self-similar sets contain many ``shrunk'' copies of
themselves, a self-similar Lie algebra is a Lie algebra containing
embedded, ``infinitesimal'' copies of itself:
\begin{defn}\label{defn:ssla}
  Let $X$ be a commutative ring with $1$, and let $\Der X$ denote the
  ring of derivations of $X$. A Lie algebra $\LL$ is
  \emph{self-similar} if it is endowed with a homomorphism
  \[\psi:\LL\to X\otimes\LL\rtimes\Der X=:\LL\wr\Der X,\]
  in which the derivations act on $X\otimes\LL$ by deriving the first
  co\"ordinate.

  If emphasis is needed, $X$ is called the \emph{alphabet} of $\LL$,
  and $\psi$ is its \emph{self-similarity structure}.
\end{defn}
If $\Bbbk$ has characteristic $p$, then $\LL$ may be a
\emph{restricted} Lie algebra, see~\S\ref{ss:notation}. Note that
$\Der X$ is naturally a restricted Lie algebra, and $X\otimes\LL$ too
for the $p$-mapping $(x\otimes a)^p=x^p\otimes a^p$. A restricted
algebra is \emph{self-similar} if furthermore $\psi$ preserves the
$p$-mapping.

We note that, under this definition, every Lie algebra is self-similar
--- though, probably, not interestingly so; indeed the definition does
not forbid $\psi=0$. The next condition makes it clear in which sense
an example in interesting or not.

A self-similar Lie algebra $\LL$ has a \emph{natural action} on the
free associative algebra $T(X)$, defined as follows: given $a\in\LL$
and an elementary tensor $v=x_1\otimes\cdots\otimes x_n\in T(X)$, set
$a\cdot v=0$ if $n=0$; otherwise, compute $\psi(a)=\sum y_i\otimes
a_i+\delta\in\LL\wr\Der X$, and set
\[a\cdot v=\sum x_1y_i\otimes(a_i\cdot x_2\otimes\cdots\otimes
x_n)+(\delta x_1)\otimes x_2\otimes\cdots\otimes x_n.\]
Finally extend the action by linearity to all of $T(X)$.

We insist that $\LL$ does \emph{not} act by derivations of $T(X)$ qua
free associative algebra, but only by endomorphisms of the underlying
$\Bbbk$-module.

A self-similar Lie algebra is \emph{faithful} if its natural action is
faithful. From now on, we will silently assume that all our Lie
algebras satisfy this condition.

There are natural embeddings $X^{\otimes n}\to X^{\otimes n+1}$, given
by $v\mapsto v\otimes1$; and for all $a\in\LL$ we have
\begin{equation}\label{eq:extendaction}
  \begin{CD}
    X^{\otimes n} @>>> X^{\otimes n+1}\\
    @VaVV @VVaV\\
    X^{\otimes n} @>>> X^{\otimes n+1}
  \end{CD}
\end{equation}
We define $R(X)=\bigcup_{n\ge0}X^{\otimes n}$ under these embeddings,
and note that $\LL$ acts on $R(X)$.

Self-similar Lie algebras may be \emph{defined} by considering
$\mathscr F$ a free Lie algebra, and $\psi:\mathscr F\to\mathscr
F\wr\Der X$ a homomorphism. The self-similar Lie algebra defined by
these data is the quotient of $\mathscr F$ that acts faithfully on
$T(X)$, namely, the quotient of $\mathscr F$ by the kernel of the
action homomorphism $\mathscr F\to\End_\Bbbk(T(X))$.

We may iterate a self-similarity structure, so as to enlarge the
alphabet: by abuse of notation we denote by $\psi^n:\LL\to X^{\otimes
  n}\otimes\LL\rtimes\Der(X^{\otimes n})=\LL\wr\Der(X^{\otimes n})$
the $n$-fold iterate of the self-similarity structure $\psi$.

\subsection{Matrix recursions}
We assume now that $X$ is finite dimensional, with basis
$\{x_1,\dots,x_d\}$. For $x\in X$, consider the $d\times d$ matrix
$m_x$ describing multiplication by $x$ on $X$. Similarly, for a
derivation $\delta\in\Der X$, consider the $d\times d$ matrix
$m_\delta$ describing its action on $X$. Consider now the
homomorphism
\begin{equation}\label{eq:matext}
\psi':\begin{cases}
  \LL\to\Mat_{d\times d}(\LL\oplus\Bbbk)\\
  a\mapsto\sum m_{y_i}a_i+m_\delta & \text{ if }\psi(a)=\sum
  y_i\otimes a_i+\delta.
\end{cases}
\end{equation}
We will see in~\S\ref{ss:assoc} that $\psi'$ extends to a
homomorphism, again written
$\psi:\U(\LL)\to\Mat_{d\times d}(\U(\LL))$, where
$\U(\LL)$ denotes the universal enveloping algebra of $\LL$. We may
therefore use this convenient matrix notation to define self-similar
Lie algebras, and study their related enveloping algebras.

\subsection{Gradings}
Assume that $X$ is graded, say by an abelian group $\Lambda$. Then, as
a module, $T(X)$ is $\Lambda[\lambda]$-graded, where the degree of
$x_1\otimes\cdots\otimes x_n$ is $\sum_{i=1}^n\deg(x_i)\lambda^{i-1}$
for homogeneous $x_1,\dots,x_n$. Here $\lambda$ is a formal parameter,
called the \emph{dilation} of the grading; though we sometimes force
it to take a value in $\R$.

If $\LL$ is both a graded Lie algebra and a self-similar Lie algebra,
its grading and self-similarity structures are compatible if, for
homogeneous $a\in\LL$, one has
$\deg(a)=\deg(y_i)+\lambda\deg(a_i)=\deg(\delta)$ for all $i$, where
$\psi(a)=\sum y_i\otimes a_i+\delta$. In other words, $\psi$ is a
degree-preserving map.

We prefer to grade the ring $X$ negatively, so that $\LL$, acting by
derivations, is graded in positive degree. This convention is of
course arbitrary.

\subsection{The full self-similar algebra}\label{ss:full}
There is a maximal self-similar algebra acting faithfully on $T(X)$
and $R(X)$, which we denote $\W(X)$. It is the set of derivations of
$T(X)$ such that~\eqref{eq:extendaction} commutes. As a vector space,
\[\W(X)=\prod_{n=0}^\infty X^{\otimes n}\otimes\Der X.\]
Its self-similarity structure is defined by
\[\psi(a_0,a_1,\dots)=\Big(\sum x_1\otimes b_1,\sum x_2\otimes b_2,\dots\Big)+a_0\]
where $a_i=\sum x_i\otimes b_i\in X\otimes\W(X)$ for all $i\ge1$. It is maximal in the
sense that every self-similar Lie algebra with alphabet $X$ is a
subalgebra of $\W(X)$. Note also that $\W(X)$ is restricted.

The Lie bracket and $p$-mapping on $\W(X)$ can be described explicitly
as follows. Given $a=x_1\otimes\cdots\otimes x_m\otimes\delta$ and
$b=y_1\otimes\cdots\otimes y_n\otimes\epsilon$, we have
\begin{align*}
  [a,b]&=\begin{cases}
  x_1y_1\otimes\cdots\otimes x_my_m\otimes\delta
  y_{m+1}\otimes\cdots\otimes y_n & \text{ if }m<n,\\
  x_1y_1\otimes\cdots\otimes x_my_m\otimes[\delta,\epsilon] &
  \text{ if }m=n,\\
  -x_1y_1\otimes\cdots\otimes x_ny_n\otimes\epsilon
  x_{n+1}\otimes\cdots\otimes x_m & \text{ if }m>n,
\end{cases}\\
\text{and }a^p&=x_1^p\otimes\cdots\otimes x_m^p\otimes\delta^p.
\end{align*}

We note that, if $X$ is a $\Lambda$-graded ring, then $\W(X)$ is
a $\Lambda[\lambda]$-graded self-similar Lie algebra; for homogeneous
$x_1,\dots,x_n\in X$ and $\delta\in\Der X$, we set
\[\deg(x_1\otimes\cdots\otimes x_n\otimes\delta)=\sum_{i=1}^n\deg(x_i)\lambda^{i-1}+\lambda^n\deg(\delta).\]

We shall consider, here, subalgebras of $\W(X)$ that satisfy
a finiteness condition. The most important is the following:
\begin{defn}\label{defn:fs}
  An element $a\in\W(X)$ is \emph{finite state} if there exists a
  finite-dimensional subspace $S$ of $\W(X)$, containing $a$, such
  that the self-similarity structure $\psi:\W(X)\to\W(X)\wr\Der X$
  restricts to a map $S\to S\otimes X\oplus\Der X$.
\end{defn}
More generally, define a self-map $\widehat\psi$ on subsets of
$\W(X)$, by
\[\widehat\psi(V)=\bigcap\{W\mid W\le\W(X)\text{ with }\psi(V)\le X\otimes W\oplus\Der X\}.\]
Then $a\in\W(X)$ is finite state if and only if
$\sum_{n\ge0}\widehat\psi^n(\Bbbk a)$ is finite-dimensional.

If $X$ is finite dimensional, a finite state element may be described
by a finite amount of data in $\Bbbk$, as follows. Choose a basis
$(e_i)$ of $S$, and write $a$ as well as the co\"ordinates of
$\psi(e_i)$ in that basis, for all $i$.

\subsection{Hausdorff dimension}\label{ss:hausdorff}
Consider a self-similar Lie algebra $\LL$, with self-similarity
structure $\psi:\LL\to\LL\wr\Der X$. As noted in~\S\ref{ss:full},
$\LL$ is a subalgebra of $\W(X)$; and both act on $X^{\otimes n}$ for
all $n\in\N$.

We wish to measure ``how much of $\W(X)$'' is ``filled in'' by
$\LL$. We essentially copy, and translate to Lie algebras, the
definitions from~\cite{bartholdi:branchalgebras}*{\S3.2}.

Let $\LL_n$, respectively $\W(X)_n$, denote the image of $\LL$,
respectively $\W(X)$, in $\Der(X^{\otimes n})$. We compute
\[\dim(\W(X)_n)=\sum_{i=0}^{n-1}(\dim X)^i\dim(\Der
X)=\frac{\dim(X)^n-1}{\dim X-1}\dim(\Der X),\]
and define the \emph{Hausdorff dimension} of $\LL$ by
\[\Hdim(\LL)=\liminf_{n\to\infty}\frac{\dim\LL_n}{\dim\W(X)_n}=\liminf_{n\to\infty}\frac{\dim\LL_n}{\dim(X)^n}\frac{\dim
  X-1}{\dim\Der X}.\]

Furthermore, there may exist a subalgebra $\mathscr P$ of $\Der X$
such that $\psi:\LL\to\LL\wr\mathscr P$; our typical examples will
have the form $X=\Bbbk[x]/(x^d)$, and $\mathscr P=\Bbbk d/dx$ a
one-dimensional Lie algebra; in that case, $\LL$ is called an
\emph{algebra of special derivations}. We then define the
\emph{relative Hausdorff dimension} of $\LL$ by
\[\Hdim_{\mathscr P}(\LL)=\liminf_{n\to\infty}\frac{\dim\LL_n}{\dim(X)^n}
\frac{\dim X-1}{\dim\mathscr P}.\]

\subsection{Bounded Lie algebras}
We now define a subalgebra of $\W(X)$, important because contains many
interesting examples, yet gives control on the nillity of the
algebra's $p$-mapping. We suppose throughout this~\S\ that $\Bbbk$ is
a ring of characteristic $p$, so that $\W(X)$ is a restricted Lie
algebra.

We suppose that $X$ is an \emph{augmented} algebra: there is a
homomorphism $\varepsilon:X\to\Bbbk$ with kernel $\varpi$. This gives
a splitting $X^{\otimes n+1}\to X^{\otimes n}$
of~\eqref{eq:extendaction}, given by $v\otimes
x_{n+1}\mapsto\varepsilon(x_{n+1})v$. The algebra $X^{\otimes n}$ also
admits an augmentation ideal,
\[\varpi_n=\ker(\varepsilon\otimes\cdots\otimes\varepsilon)=\sum X\otimes\cdots\otimes\varpi\otimes\cdots\otimes X.\]
The union of the $\varpi_n$ defines an augmentation ideal again
written $\varpi$ in $R(X)$.

There is a natural action of $R(X)$ on $\W(X)$: given
$a=x_1\otimes\cdots\otimes x_m\in R(X)$ and
$b=y_1\otimes\cdots\otimes y_n\otimes\delta\in \W(X)$, first replace
$a$ by $a\otimes1\otimes\cdots\otimes1$ with enough $1$'s so that
$m\ge n$; then set
\[a\cdot b=\varepsilon(y_{n+1})\cdots\varepsilon(y_m)x_1y_1\otimes\cdots\otimes x_ny_n\otimes\delta.\]

\begin{defn}
  An element $a\in \W(X)$ is \emph{bounded} if there exists a constant
  $m$ such that $\varpi^ma=0$. Writing $a=(a_0,a_1,\dots)$, this means
  $\varpi^ma_i=0$ for all $i$. The set of bounded elements is written
  $M(X)$.

  The \emph{bounded norm} $\|a\|$ of $a$ is then the minimal such $m$.
\end{defn}

\noindent The following statement is inspired
by~\cite{shestakov-z:nillie}*{Lemma~1}.
\begin{lem}\label{lem:bounded}
  The set $M=M(X)$ of bounded elements forms a restricted Lie
  subalgebra of $\W(X)$.

  More precisely, the bounded norm of $[a,b]$ is at most
  $\max\{\|a\|,\|b\|\}+1$, and $\|a^p\|\le\|a\|+p-1$.
\end{lem}
\begin{proof}
  If $\varpi^mf=0$ and $\delta\in\Der X$, then
  $\varpi^{m+1}\delta f=0$. It follows that, for $a\in M$ of
  bounded norm $m$ and $b\in M$ of bounded norm $n$, the bounded norm
  of $[a,b]$ is at most $\max\{m,n\}+1$.

  Consider next $a=(a_0,a_1,\dots)\in M$, of bounded norm $m$. Write
  each $a_i=\sum f_i\otimes \delta_i$. Then $a_i^p=\sum
  f_i^p\otimes \delta_i^p+$commutators of weight $p$; now
  $\varpi^mf_i^p\otimes\delta_i^p=0$ and, by the first paragraph,
  commutators are annihilated by $\varpi^{m+p-1}$.
\end{proof}

We now suppose for simplicity that the alphabet has the form
$X=\Bbbk[x]/(x^p)$. Its augmentation ideal is $\varpi=xX$, and
satisfies $\varpi^p=0$. We seek conditions on elements $a\in M(X)$
that ensure that they are \emph{nil}, that is, there exists
$n\in\N$ such that $a^{p^n}=(((a^p)^p)\cdots)^p=0$. The standard
derivation $d/dx$ of $X$ is written $\partial_x$.

\begin{defn}
  An element $a\in \W(X)$ is \emph{$\ell$-evanescent}, for $\ell\in\N$,
  if when we write $a=(a_0,a_1,\dots)$, each $a_i$ has the form $\sum
  b_i\otimes c_i\otimes \partial_x$ with $b_i\in
  X^{\otimes\max\{i-\ell,0\}},c_i\in X^{\otimes\min\{i,\ell\}}$, and
  $\deg(c_i)<(p-1)\ell$.

  An element is \emph{evanescent} if it is $\ell$-evanescent for some
  $\ell\in\N$.
\end{defn}
In words, $a$ is $\ell$-evanescent if, in all co\"ordinates $a_i$ of
$a$, the derivation is $\partial_x$ and the maximal degree is
\emph{never} reached in each of the last $\ell$ alphabet variables.

\begin{lem}\label{lem:evanescent}
  The set of $\ell$-evanescent elements forms a restricted Lie
  subalgebra of $\W(X)$.
\end{lem}
\begin{proof}
  Consider $\ell$-evanescent elements $a,b\in \W(X)$, and without loss
  of generality assume $a=x_1\otimes\cdots\otimes
  x_m\otimes\partial_x$, $b=y_1\otimes\cdots\otimes
  y_n\otimes\partial_x$ and $m\le n$. If $m=n$, then $[a,b]=0$, while
  if $m<n$ then $[a,b]=x_1y_1\otimes\cdots\otimes \partial_x
  y_{m+1}\otimes\cdots\otimes y_n\otimes\partial_x$. If $n-m>\ell$
  then there is nothing to do, while if $n-m\le\ell$ then
  $\deg(\delta y_{m+1})<p-1$ so the total degree in the last $\ell$
  alphabet variables of $[a,b]$ is $<(p-1)\ell$.

  Finally, the $p$-mapping is trivial on elementary tensors
  $x_1\otimes\cdots\otimes x_n\otimes\partial_x$ because
  $\partial_x^p=0$ in $\Der X$.
\end{proof}

\noindent The following statement is inspired
by~\cite{shestakov-z:nillie}*{Lemma~2}.
\begin{lem}\label{lem:evanescent2}
  Let $a\in M(X)$ be evanescent. Then there exists $s\ge1$ so that
  $a^{p^s}\in\varpi^2\W(X)$.
\end{lem}
\begin{proof}
  We put an $\R$-grading on $\W(X)$, by using the natural
  $\{1-p,\dots,0\}$-grading of $X$ and choosing for $\lambda$ the
  largest positive root of
  $f(\lambda)=\lambda^{\ell+1}-p\lambda^\ell+\lambda-1$. We note that
  $f(+\infty)=+\infty$ and $f(0)=-1<0$, so $f$ has one or three
  positive roots. Next, $f'(0)>0$ and $f'$ has at most two positive
  roots in $\R_+$, while $f(1)=1-p\le f(0)$, so $f$ has at most one
  extremum in $(1,\infty)$. Finally, $f(p)=p-1>0$, so $f$ has a unique
  zero in $(1,p)$ and we deduce $\lambda\in(1,p)$.

  Consider a homogeneous component $h=x_1\otimes\cdots\otimes
  x_n\otimes\partial_x$ of $a$. Because $\deg(\partial_x)=1$ and
  $-\deg(x_i)\le p-1$ for all $i$ and $-\deg(x_m)\le p-2$ for some
  $m\in\{n-\ell+1,\dots,n\}$, we have
  \begin{align*}
    \deg(h)&\ge\lambda^n-(p-1)(1+\lambda+\cdots+\lambda^{n-1})+\lambda^m\\
    &\ge\lambda^n-(p-1)\frac{\lambda^n-1}{\lambda-1}+\lambda^{n-\ell}=\frac{\lambda^{n-\ell}}{\lambda-1}f(\lambda)+\frac{p-1}{\lambda-1}\\
    &=\frac{p-1}{\lambda-1}>1;
  \end{align*}
  so every homogeneous component of $a^{p^s}$ has degree $\ge
  p^s$.

  We now use the assumption $a\in M(X)$, say $\|a\|=m$. Then
  $\|a^{p^s}\|\le m+(p-1)s$, by Lemma~\ref{lem:bounded}. Write
  $a^{p^s}=(b_0,b_1,\dots)\in \W(X)$; we therefore have
  $\varpi_i^{m+(p-1)s}b_i=0$ for all $i\ge0$.

  Consider some $j\le i$, and assume that $b_i$ does not belong to
  $\varpi_j^2X^{\otimes i}\otimes\Der X$. Then
  $\varpi_i^{(p-1)j-1}b_i\neq0$, and this can only happen if
  $(p-1)j-1<m+(p-1)s$. In other words, for every $s\in\N$ there exists
  $j\in\N$ such that $b_i\in\varpi_j^2\otimes X^{\otimes i-j}$ for all
  $i\ge j$.

  Consider then the $b_i$ with $i<j$. The degree of a homogeneous
  component in such a $b_i$ is at most $\lambda^j$; on the other hand,
  since it is a homogeneous component of $a^{p^s}$, it has degree at
  least $p^s$. Now the inequalities
  \[p^s\le\lambda^j,\qquad (p-1)j-1<m+(p-1)s\]
  cannot simultaneously be satisfied for arbitrarily large $s$,
  because $\lambda<p$. It follows that, at least for $s$ large enough,
  $b_i=0$ for all $i<j$, and therefore $a^{p^s}\in\varpi_j^2M(X)$.
\end{proof}

To make this text self-contained, we reproduce almost \emph{verbatim}
the following statement by Shestakov and Zelmanov:
\begin{lem}[\cite{shestakov-z:nillie}*{Lemma~5}]\label{lem:sz}
  Assume that $\varpi\subset X$ is nilpotent. Then the associative
  subalgebra of $\End_\Bbbk T(X)$ generated by $\varpi^2\W(X)$ is locally
  nilpotent.
\end{lem}
\begin{proof}
  By assumption $X^d=0$ for some $d\in\N$. Consider a finite
  collection of elements $a_i=x'_ix''_ib_i\in\varpi^2\W(X)$ for
  $i\in\{1,\dots,r\}$, with $x'_i,x''_i\in\varpi$ and $b_i\in
  \W(X)$. Let $A$ be the subalgebra of $\varpi$ generated by
  $x'_1,x''_1,\dots,x'_r,x''_r$. From $(x'_i)^d=(x''_i)^d=0$ follows $A^s=0$
  with $s=2r(d-1)+1$. Now $a_{i_1}\cdots a_{i_s}=\sum y_1\cdots
  y_{2s}d_{j_1}\cdots d_{j_q}$, with $q\le s$ and the $y_j$ obtained
  from the $x'_i,x''_i$ by $(q-s)$-fold application of the derivations
  $b_k$. Therefore $2s-q\ge s$ of the $y_j$ belong to
  $\{x'_1,x''_1,\dots,x'_r,x''_r\}$, so $a_{i_1}\cdots a_{i_s}=0$.
\end{proof}

\begin{cor}\label{cor:nil}
  If $\LL$ is a subalgebra of $\W(X)$ that is generated by bounded,
  $\ell$-evanescent elements for some $\ell\in\N$, then $\LL$ is nil.
\end{cor}
\begin{proof}
  It follows from Lemmata~\ref{lem:bounded} and~\ref{lem:evanescent}
  that every element $a\in\LL$ is bounded and $\ell$-evanescent; and
  then from Lemmata~\ref{lem:evanescent2} and~\ref{lem:sz} that $a$ is
  nil.
\end{proof}

\subsection{Growth and contraction}
Let $\AA$ denote an algebra, not necessarily associative. For a
finite dimensional subspace $S\le\AA$, let $S^n$ denote the span in
$\AA$ of $n$-fold products of elements of $S$; if $p|n$ and $\AA$ is a
restricted Lie algebra in characteristic $p$, then $S^n$ also contains
the $p$th powers of elements of $S^{n/p}$. Define
\[\GKdim(\AA,S)=\limsup\frac{\log\dim(S^n)}{\log n},\qquad\GKdim(\AA)=\sup_{S\le\AA}\GKdim(\AA,S),\]
the (upper) \emph{Gelfand-Kirillov dimension} of $\AA$. Note that if
$\AA$ is generated by the finite dimensional subspace $S$, then
$\GKdim(\AA)=\GKdim(\AA,S)$. We give here conditions on a self-similar
Lie algebra $\LL$ so that it has finite Gelfand-Kirillov dimension.

\begin{defn}
  Let $\LL$ be a self-similar Lie algebra, with self-similarity
  structure $\psi:\LL\to\LL\wr\Der X$. It is \emph{contracting} if
  there exists a finite dimensional subspace $N\le\LL$ with the
  following property: for every $a\in\LL$, there exists $n_0\in\N$
  such that, for all $n\ge n_0$, we have $\psi^n(a)\in X^{\otimes
    n}\otimes N\oplus\Der(X^{\otimes n})$.

  The minimal such $N$, if it exists, is called the \emph{nucleus} of
  $\LL$.
\end{defn}
In words, the nucleus is the minimal subspace of $\LL$ such that, for
every $a\in\LL$, if one applies often enough the map $\psi$ to it and
its co\"ordinates (discarding the term in $\Der X$), one obtains only
elements of $N$.

Note that elements of a contracting self-similar algebra are
finite-state. The following test is useful in practice to prove that
an algebra is contracting, and leads to a simple algorithm:
\begin{lem}
  Let the self-similar Lie algebra $\LL$ be generated by the
  finite dimensional subspace $S$, and consider a finite dimensional
  subspace $N\le\LL$. Then $N$ contains the nucleus of $\LL$ if and
  only if there exist $m_0,n_0\in\N$ such that $\psi^m(S+N+[N,S])\le
  X^{\otimes m}\otimes N\oplus\Der(X^{\otimes m})$ and (for restricted
  algebras) $\psi^n((N+S)^p)\le X^{\otimes n}\otimes
  N\oplus\Der(X^{\otimes n})$ for all $m\ge m_0,n\ge n_0$.
\end{lem}
\begin{proof}
  Note first that, if $N$ contains the nucleus $N_0$, then
  co\"ordinates of $\psi^m([N+S,S])$ and $\psi^n((N+S)^p)$ will be
  contained in $N_0$, so \emph{a fortiori} in $N$ for all sufficiently
  large $m,n$.

  Conversely, consider $a\in\LL$, say a product of $\ell$ elements of
  $S$. Then applying $\ell-1$ times the map $\psi^m$, we get
  $\psi^{m(\ell-1)}(a)\in X^{\otimes m(\ell-1)}\otimes
  N+\Der(X^{\otimes m(\ell-1)})$, so $N$ contains the nucleus.
\end{proof}

\begin{algorithm}[H]
\DontPrintSemicolon
\KwData{A generating set $S$}
\KwResult{The nucleus $N$}
\SetAlgoLined
%\SetLine
$N\leftarrow0$ \;
\Repeat{$N=N'$}{$N'\leftarrow N$ \;
  $B\leftarrow N+S+[N,S]$ \;
  $N\leftarrow\bigcap_{i\ge0}\sum_{j\ge i}\widehat\psi^j(B)$ \;}
\end{algorithm}

\begin{defn}
  Let $\LL$ be a self-similar algebra; assume that the alphabet $X$
  admits an augmentation $\varepsilon:X\to\Bbbk$, and denote by $\pi$
  the projection $\LL\to\Der X$. We say $\LL$ is \emph{recurrent} if
  the $\Bbbk$-linear map $(\varepsilon\otimes1)\psi: \ker(\pi)\to\LL$
  is onto.
\end{defn}

\begin{lem}
  Let $\LL$ be a finitely generated, contracting and recurrent
  self-similar Lie algebra. Then $\LL$ is generated by its nucleus.
\end{lem}
\begin{proof}
  Let $S$ be a generating finite dimensional subspace, and let $N$
  denote the nucleus of $\LL$. Because $S$ is finite dimensional,
  there exists $n\in\N$ such that $\psi^n(S)\le X^{\otimes n}\otimes
  N\oplus\Der(X^{\otimes n})$. Let $\mathscr M$ denote the subalgebra
  generated by $N$; then, for every $a\in\LL$, we have
  $\psi^n(a)\in\mathscr M\wr\Der(X^{\otimes n})$; so, because $\LL$ is
  recurrent, $\mathscr M=\LL$.
\end{proof}

\begin{lem}\label{lem:sum}
  Let $\LL$ be a contracting, finitely generated self-similar Lie
  algebra, with $X$ finite dimensional. Then there exists a finite
  dimensional generating subspace $N$ of $\LL$ such that
  \[\LL\le\sum_{n\ge0}X^{\otimes n}\otimes N.\]
\end{lem}
\begin{proof}
  Let $N_0$ be the nucleus of $\LL$, and let $S$ generate
  $\LL$. Enlarge $S$, keeping it finite dimensional, so that
  $\pi(N)=\pi(\LL)\le\Der X$. Set finally
  $N=N_0+\sum_{n\ge0}\widehat\psi(S)$.

  By the definition of nucleus, $\LL$ is contained in
  $\sum_{n\ge0}X^{\otimes n}\otimes N\oplus\Der(X^{\otimes n})$. Now
  using the fact that $\pi(N)=\pi(\LL)$, we can eliminate the
  $\Der(X^{\otimes n})$ terms while still staying in $X^{\otimes
    n}\otimes N$.
\end{proof}

\begin{lem}
  Let $\LL$ be a contracting, $\R_+$-graded self-similar Lie algebra,
  with dilation $\lambda$, and let $N$ be as in
  Lemma~\ref{lem:sum}. Consider a homogeneous $a\in\LL$ with
  $\deg(a)=d$. Then $a\in\sum_{j=0}^nX^{\otimes j}\otimes N$ with
  \begin{equation}\label{eq:contrub}
    n\ge\log(d/m)/\log|\lambda|\text{ where }m=\max_{n\in N}\deg(n).
  \end{equation}

  If furthermore $\lambda>1$ and $\LL$ is generated by finitely many
  positive-degree elements, then there exists $\epsilon>0$ such
  that every $a\in\LL$ satisfies
  \begin{equation}\label{eq:contrlb}
    n\le \log(d/\epsilon)/\log\lambda.
  \end{equation}
\end{lem}
\begin{proof}
  Consider $a\in X^{\otimes n}\otimes N\cap\LL$. Then
  $\deg(a)\le\lambda^n\max\deg(N)$, and this proves~\eqref{eq:contrub}.

  Now, if $\LL$ is generated by finitely many positive-degree
  elements, then there exists $m\in\N$ such that all elements of
  $X^{\otimes m}\otimes N\cap\LL$ have degree $> -\min_{x\in
    X}\deg(x)/(\lambda-1)$. Let then $\epsilon>0$ be such that all
  these elements have degree $\ge B:=\lambda^m\epsilon-\min_{x\in
    X}\deg(x)/(\lambda-1)$.

  We return to our $a$, which we write as $a=\sum f\otimes b$, with
  $f\in X^{\otimes m-n}$ and $b\in X^{\otimes m}\cap\LL$. Then every
  homogeneous summand of $b$ has degree at least $B$, while every
  homogeneous summand of $f$ has degree at least $\min_{x\in
    X}\deg(x)(1+\lambda+\dots+\lambda^{m-n-1}$; so every homogeneous
  summand of $a$ has degree at least $\epsilon\lambda^n$.
\end{proof}

\begin{prop}\label{prop:GKupper}
  Let $\LL$ be a contracting, $\R_+$-graded self-similar Lie algebra,
  with dilation $\lambda>1$, that is generated by finitely many
  positive-degree elements.

  Then $\LL$ has finite Gelfand-Kirillov dimension; more precisely,
  \[\GKdim(\LL)\le\frac{\log(\dim X)}{\log\lambda}.\]
\end{prop}
\begin{proof}
  Let $\LL_d$ denote the span of homogeneous elements in $\LL$ of
  degree $\le d$. Consider $a\in\LL_d$, and, up to expressing $a$ as a
  sum, assume $a=f\otimes b$ with $f\in X^{\otimes n}$ and $b\in
  N$. By~\eqref{eq:contrlb} we have $n\le\log_\lambda(d/\epsilon)$, so
  \[\dim\LL_d\le\sum_{j=0}^n(\dim X)^j\dim N\precapprox d^{\log(\dim X)/\log\lambda}.\qedhere\]
\end{proof}

%%%%%%%%%%%%%%%%%%%%%%%%%%%%%%%%%%%%%%%%%%%%%%%%%%%%%%%%%%%%%%%%
\section{(Weakly) branched Lie algebras}\label{ss:wbr}
We will concentrate here on some conditions on a self-similar Lie
algebra that have consequences on its algebraic structure, and in
particular on its possible quotients. We impose, in this~\S,
restrictions that will be satisfied by all our examples: the alphabet
$X$ has the form $\Bbbk[x]/(x^d)$, and the image of $\LL$ in $\Der X$
is precisely $\Bbbk\partial_x$. Let $\theta=x^{d-1}$ denote the
top-degree element of $X$.

We consider self-similar Lie algebras $\LL$ with self-similarity
structure $\psi:\LL\to\LL\wr\Bbbk\partial_x$. Let
$\pi:\LL\to\Bbbk\partial_x$ denote the natural projection. We recall
that $\LL$ acts on $X^{\otimes n}$ for all $n$. We denote by $\U(\LL)$
the universal enveloping algebra of $\LL$; then $\U(\LL)$ also acts on
$X^{\otimes n}$, and acts on $\LL$ by derivations. If $u=u_1\cdots
u_n\in\U(\LL)$ and $a\in\LL$, we write
$[[u,a]]=[u_1,[u_2,\dots,[u_n,a]\cdots]]$ for this action.

We also identify $\LL$ with $\psi(\LL)$; so as to write, e.g.,
`$\theta\otimes a\in\LL$' when we mean `$\theta\otimes
a\in\psi(\LL)$'.

\begin{defn}
  The self-similar algebra $\LL$ is \emph{transitive} if for all $v\in
  X^{\otimes n}$ there exists $u\in\U(\LL)$ with
  $u\cdot(\theta^{\otimes n})=v$.
\end{defn}
We note immediately that, if $\LL$ is transitive, then it is
infinite-dimensional.

\begin{lem}\label{lem:transitive}
  If $\U(\LL)\cdot\theta=X$ and $\LL$ is recurrent, then $\LL$ is
  transitive.
\end{lem}
\begin{proof}
  We proceed by induction on $n$, the case $n=0$ being obvious. Write
  $M=\U(\LL)\cdot\theta^{\otimes n+1}$. Because $\LL$ is recurrent and
  by the inductive hypothesis, $\U(\ker\pi)\cdot\theta^{\otimes
    n+1}=\theta\otimes X^{\otimes n}$. Then, by hypothesis, there
  exists in $\LL$ an element $a$ of the form $\partial_x+\sum
  x_i\otimes a_i$; so $a\cdot\theta^{\otimes n+1}\in
  x^{d-2}\otimes\theta^{\otimes n}+\theta\otimes X^{\otimes
    n}$. Because $M$ contains $\theta\otimes X^{\otimes n}$, it also
  contains $x^{d-2}\otimes\theta^{\otimes n}$, and again acting with
  $\U(\ker\pi)$ it contains $x^{d-2}\otimes X^{\otimes n}$. Continuing
  in this manner, it contains $x^i\otimes X^{\otimes n}$ for all $i$,
  and therefore equals $X^{\otimes n+1}$.
\end{proof}

\begin{defn}
  Let $\LL$ be a recurrent, transitive, self-similar Lie algebra. It
  is
  \begin{description}
  \item[weakly branched] if for every $n\in\N$ there exists a non-zero
    $a\in \W(X)$ such that $\theta^{\otimes n}\otimes a\in\LL$;
  \item[branched] if for every $n\in\N$ the ideal $K_n$ generated by
    all $\theta^{\otimes n}\otimes a\in\LL$ has finite codimension in
    $\LL$;
  \item[regularly weakly branched] if there exists a non-trivial ideal
    $\KK\triangleleft\LL$ such that $X\otimes\KK\le\psi(\KK)$;
  \item[regularly branched] if furthermore there exists such a $\KK$
    of finite codimension in $\LL$.
  \end{description}

  In the last two cases, we say that $\LL$ is regularly [weakly]
  branched \emph{over} $\KK$.
\end{defn}

\noindent The following immediately follows from the definitions:
\begin{lem}
  ``Regularly branched'' implies ``Regularly weakly branched'' and
  ``branched'', and each of these implies ``weakly branched''.\qedhere
\end{lem}

Note, as a partial converse, that if $\LL$ is weakly branched, then
\[\{a\in\LL\mid \theta^{\otimes n}\otimes a\le\LL\}=
\{a\in\LL\mid X^{\otimes n}\otimes a\le\LL\}=:\KK_n\]
is a non-trivial ideal in $\LL$:
\begin{lem}
  If $\LL$ is weakly branched, then for every $v\in X^{\otimes n}$
  there exists a non-zero $a\in \W(X)$ with $v\otimes a\in\LL$.
\end{lem}
\begin{proof}
  Let $0\neq a\in \W(X)$ be such that $\theta^{\otimes n}\otimes
  a\in\LL$; let $\KK_n$ denote the ideal generated by
  $a$. Because $\LL$ is recurrent, $\theta^{\otimes n}\otimes c\in\LL$
  for all $c\in\KK$. Because $\LL$ is transitive, there exists
  $u\in\U(\LL)$ with $u\cdot\theta^{\otimes n}=v$. Consider then
  $[[u,\theta^{\otimes n}\otimes a]]$. It is of the form $v\otimes
  a+\sum v'\otimes a'$, for some $a'\in\KK$ and $v'>v$ in reverse
  lexicographic ordering. By induction on $v$ in that ordering,
  $v\otimes a$ belongs to $\LL$.
\end{proof}

We are now ready to deduce some structural properties of (weakly)
branched Lie algebras:
\begin{prop}
  Let $\LL$ be a weakly branched Lie algebra. Then the centralizer of
  $\LL$ in $\W(X)$ is trivial. In particular, $\LL$ has trivial centre.
\end{prop}
\begin{proof}
  Consider a non-zero $a\in \W(X)$. There then exists $v\in X^{\otimes
    n}$ such that $a\cdot v\neq0$; suppose that $n$ is minimal for
  that property. Let $i\in\{0,\dots,d-1\}$ be maximal such that
  $1^{\otimes n-1}\otimes\varpi^i(a\cdot v)\neq0$. Since $\LL$ is
  weakly branch, there exists a non-zero element $b=1^{\otimes
    n-1}\otimes x^{i+1}\otimes b'\in\LL$; and because $b\neq0$, there
  exists $w\in X^{\otimes m}$ with $b\cdot x\neq0$. Consider
  $c=[a,b]$; the claim is that $c\neq0$. Indeed $c\cdot(v\otimes
  w)=a\cdot b\cdot(v\otimes w)-b\cdot a\cdot(vw)$; and
  $b\cdot(v\otimes w)=0$, while $b\cdot a\cdot(v\otimes w)=(a\cdot
  v)\otimes(b\cdot w)\neq0$.
\end{proof}

Recall that a (non-necessarily associative) algebra $\AA$ is PI
(``Polynomial Identity'') if there exists a non-zero polynomial
expression $\Psi(X_1,\dots,X_n)$ in non-associative, non-commutative
indeterminates such that $\Psi(a_1,\dots,a_n)=0$ for all $a_i\in\AA$.
\begin{prop}\label{prop:notPI}
  Let $\LL$ be a weakly branched Lie algebra. Then $\LL$ is not PI.
\end{prop}
\begin{proof}
  There should exist a purely Lie-theoretical proof of this fact, but
  it is shorter to note that if $\LL$ is weakly branched, then its
  associative envelope $\AA$ (see~\S\ref{ss:assoc}) is weakly branched
  in the sense of~\cite{bartholdi:branchalgebras}*{\S3.1.6}. Weakly
  branched associative algebras satisfy no polynomial identity
  by~\cite{bartholdi:branchalgebras}*{Theorem~3.10}, so the same must
  hold for any Lie subalgebra that generates $\AA$.
\end{proof}

Recall that a Lie algebra is \emph{just infinite} if it is infinite
dimensional, but all its proper quotients are finite dimensional.
\begin{prop}\label{prop:ji}
  Let $\LL$ be a regularly branched Lie algebra, with branching ideal
  $\KK$. If furthermore $\KK/[\KK,\KK]$ is finite-dimensional, then
  $\LL$ is just infinite.
\end{prop}
(Note that the proposition's conditions are clearly necessary;
otherwise, $\LL/[\KK,\KK]$ would itself be a finite-dimensional proper
quotient).
\begin{proof}
  Consider a non-zero ideal $\II\triangleleft\LL$. Without loss
  of generality, $\II=\langle a\rangle$ is principal. Let $n\in\N$ be
  minimal such that $a\cdot X^{\otimes n}\neq0$; so
  $a=x^{i_1}\otimes\cdots\otimes x^{i_{n-1}}\otimes\partial_x+$
  higher-order terms. Because $\LL$ is transitive, we can derive
  $i_1+\cdots+i_{n-1}$ times $a$, by an element $u\in\U(\LL)$, to
  obtain $b=[[u,a]]\in\II$ of the form
  \[b=1^{\otimes n-1}\otimes\partial_x+1^{\otimes n}\otimes
  b'+\text{higher-order terms}.\]

  Consider now $c=\theta^n\otimes k\in\KK$. Then
  $[b,c]=\theta^{n-1}\otimes x^{d-2}\otimes
  k+\theta^n\otimes[b',k]+$higher-order terms.

  Consider next $d=1^{\otimes n-1}\otimes x\otimes\ell\in\KK$. Then
  $[[b,c],d]=\theta^n\otimes[k,\ell]\in\II$.

  It follows that $\II$ contains $X^{\otimes n}\otimes[\KK,\KK]$, and
  therefore that $\LL/\II$ has finite dimension.
\end{proof}

\begin{cor}\label{cor:ji}
  Let $\LL$ be a finitely generated, nil, regularly branched Lie
  algebra. Then $\LL$ is just infinite.
\end{cor}
\begin{proof}
  Let $\KK$ denote the branching ideal of $\LL$.  Since $\LL$ is
  finitely generated and $\KK$ has finite index in $\LL$, it is also
  finitely generated~\cite{kukin:subalgebras}. Since $\LL$ is nil, the
  abelianization of $\KK$ is finite and we may apply
  Proposition~\ref{prop:ji}.
\end{proof}

\begin{prop}\label{prop:unboundednil}
  Let $\LL$ be a regularly branched Lie algebra. Then there does not
  exist a bound on its nillity.
\end{prop}
\begin{proof}
  Assume that $\KK$ contains a non-trivial nil element, say $a\in\KK$
  with $a^{p^m}=0$ but $a^{p^{m-1}}\neq0$. Let $b\in\KK$ be such that
  $b=1^{\otimes\ell}\otimes\partial_x+$higher terms. Construct then
  the following sequence of elements: $a_0=a$, and
  $a_{n+1}=1^{\otimes\ell-1}\otimes\theta\otimes a_n+b$. By induction,
  the element $a_n$ has nillity exactly $p^{m+n}$.
\end{proof}

The following is essentially~\cite{bartholdi:branchalgebras}*{Proposition~3.5}:
\begin{prop}\label{prop:hdimrat}
  Let $\LL$ be a regularly branched Lie algebra. Then its Hausdorff
  dimension is a rational number in $(0,1]$.
\end{prop}
\begin{proof}
  Suppose that $\LL$ is regularly branched over $\KK$.  As
  in~\S\ref{ss:hausdorff}, let $\LL_n$ denote the image of $\LL$ in
  $\Der(X^{\otimes n})$, with quotient map $\pi_n:\LL\to\LL_n$. Let
  $M$ be large enough so that $\LL/\psi^{-1}(X\otimes\KK)$
  maps isomorphically onto its image in $\LL_n$. We have, for all
  $n\ge M$,
  \begin{align*}
    \dim \LL_n&=\dim(\LL/\KK)+\dim\pi_n(\KK)\\
    &=\dim(\LL/\KK)+\dim(\psi\KK/(X\otimes\KK))+\dim(X)\dim\pi_{n-1}(\KK)\\
    &=(1-\dim X)\dim(\LL/\KK)+\dim(\psi\KK/(X\otimes\KK))+\dim X\dim\LL_{n-1}.
  \end{align*}
  We write $\dim\LL_n=\alpha\dim(X)^n+\beta$, for some $\alpha,\beta$ to
  be determined; we have
  \begin{multline*}
    \alpha\dim(X)^n+\beta=(1-\dim X)\dim(\LL/\KK)+\dim(\psi\KK/(X\otimes\KK))\\
    +\dim(X)(\alpha\dim(X)^{n-1}+\beta),
  \end{multline*}
  so $\beta=\dim(\LL/\KK)-\dim(\psi\KK/(X\otimes\KK))/(\dim X-1)$. Then
  set $\alpha=(\dim \LL_M-\beta)/\dim(X)^M$. We have solved the
  recurrence for $\dim \LL_n$, and $\alpha>0$ because $\LL_n$ has
  unbounded dimension, since $\LL$ is infinite-dimensional.

  Now it suffices to note that $\Hdim(\LL)=\alpha$ to obtain
  $\Hdim(\LL)>0$. Furthermore only linear equations with integer
  coefficients were involved, so $\Hdim(\LL)$ is rational.
\end{proof}

Note that, if $\LL\to\LL\wr\mathscr P$ is the self-similarity
structure, then $\Hdim_{\mathscr P}(\LL)$ is also positive and
rational.

\begin{prop}\label{prop:GKlower}
  Let $\LL$ be a regularly weakly branched self-similar Lie
  algebra. Suppose $\LL$ is graded, with dilation $\lambda>1$. Then
  the Gelfand-Kirillov dimension of $\LL$ is at least $\log(\dim
  X)/\log\lambda$.
\end{prop}
\begin{proof}
  Suppose $\LL$ is weakly branched over $\KK$; consider a non-zero
  $a\in\KK$. Let $\epsilon$ be the degree of $a$. Then $\LL$ contains
  for all $n\in\N$ the subspace $X^{\otimes n}\otimes\Bbbk a$; and the
  maximal degree of these elements is $\lambda^n\epsilon$. Let $\LL_d$
  denote the span of elements of degree $\le d$; it follows that
  $\dim\LL_d\ge(\dim X)^n$ whenever $\lambda^n\epsilon\le d$, and therefore
  \[\dim\LL_d\ge (d/\epsilon)^{\log(\dim X)/\log\lambda}.\qedhere\]
\end{proof}

%%%%%%%%%%%%%%%%%%%%%%%%%%%%%%%%%%%%%%%%%%%%%%%%%%%%%%%%%%%%%%%%
\section{Examples}\label{ss:examples}
We begin by two examples of Lie algebras, inspired and related to
group-theoretical constructions. The links between the groups and the
Lie algebras will be explored in~\S\ref{ss:ssg}. We then phrase in
our language of self-similar algebras an example by Petrogradsky,
later generalized by Shestakov and Zelmanov.

\subsection{The Gupta-Sidki Lie algebra}\label{ss:gsla}
Inspired by the self-similarity structure~\eqref{eq:gsgp}, we consider
$X=\mathbb F_3[x]/(x^3)$ and a Lie algebra $\LLgs=\LL$ generated by $a,t$
with self-similarity structure
\begin{equation}\label{eq:gsla}
  \psi:\begin{cases}
    \LL\to\LL\wr\Der X\\
    a\mapsto\partial_x\\
    t\mapsto x\otimes a+x^2\otimes t.
  \end{cases}
\end{equation}
We put a grading on $\LL$ such that the generators $a,t$
are homogeneous. The ring $X$ is $\Z$-graded, with $\deg(x)=-1$ so
$\deg(a)=1$, and
\[\deg(t)=-2+\lambda\deg(t)=-1+\lambda\deg(a),\]
so $(1-\lambda)^2=2$ and $\deg(t)=\sqrt2$; the Lie algebra $\LL$ is
$\Z[\lambda]/(\lambda^2-2\lambda-1)$-graded.
We repeat the construction using our matrix notation. For that
purpose, we take divided powers $\{1,t,t^2/2=-t^2\}$
as basis of $X$; the self-similarity structure is then
\[a\mapsto\begin{pmatrix} 0 & 1 & 0\\ 0 & 0 & 1\\ 0 & 0 & 0\end{pmatrix},\quad
  t\mapsto\begin{pmatrix} 0 & 0 & 0\\ a & 0 & 0\\ -t & -a & 0\end{pmatrix}.\]

\begin{prop}
  The Lie algebra $\LL$ is regularly branched on its ideal $[\LL,\LL]$
  of codimension $2$.
\end{prop}
\begin{proof}
  First, $\LL$ is recurrent: indeed $(\varepsilon\otimes1)\psi[a,t]=a$
  and $(\varepsilon\otimes1)\psi[a,[a,t]]=t$. Then, by
  Lemma~\ref{lem:transitive}, $\LL$ is transitive.

  The ideal $[\LL,\LL]$ is generated by $c=[a,t]$; to prove that $\LL$
  is branched on $[\LL,\LL]$, it suffices to exhibit $c'\in[\LL,\LL]$
  with $\psi(c')=x^2\otimes c$. A direct calculation shows that
  $c'=[[a,t],t]$ will do.

  Clearly $\LL/[\LL,\LL]$ is the commutative algebra $\Bbbk^2$
  generated by $a,t$.
\end{proof}

\begin{thm}
  The Lie algebra $\LL$ is nil, of unbounded nillity, but not nilpotent.
\end{thm}
\begin{proof}
  The ideal $\langle t\rangle$ in $\LL$ has codimension $1$; and it is
  generated by $t,[t,a],[t,a,a]$. Each of these elements is bounded
  and $1$-evanescent; Corollary~\ref{cor:nil} applies. Then $\LL$
  itself is nil, because $a^3a=0$.

  That the nillity in unbounded follows from
  Proposition~\ref{prop:unboundednil}. Clearly $\LL$ is not nilpotent,
  since by Proposition~\ref{prop:notPI} it is not even PI.
\end{proof}

\begin{prop}
  The relative Hausdorff dimension of $\LL$, with $\mathscr
  P=\Bbbk\partial_x$, is
  \[\Hdim_{\mathscr P}(\LL)=\frac49.\]
\end{prop}
\begin{proof}
  We follow Proposition~\ref{prop:hdimrat}. We may take $M=2$, and
  readily compute $\dim(\LL/\KK)=2=\dim(\KK/(X\otimes\KK))$, the
  latter having basis $\{[a,t],[a,[a,t]]\}$. Letting $\LL_n$ denote
  the image of $\LL$ in $\Der(X^{\otimes n})$, we find
  $\dim(\LL_2)=3$. This gives
  \[\dim(\LL_n)=2\cdot3^{n-2}+1,\]
  and the claimed result.
\end{proof}

\begin{lem}
  Let algebra $\LL$ is contracting.
\end{lem}
\begin{proof}
  Its nucleus is $\Bbbk\{a,t\}$.
\end{proof}

\begin{cor}
  The Gelfand-Kirillov dimension of $\LL$ is $\log3/\log\lambda$.
\end{cor}
\begin{proof}
  This follows immediately from Propositions~\ref{prop:GKupper}
  and~\ref{prop:GKlower}.
\end{proof}

In fact, thanks to Theorem~\ref{thm:isogs}, a much stronger result holds:
\begin{prop}[\cite{bartholdi:lcs}*{Corollary~3.9}]
  Set $\alpha_1=1$, $\alpha_2=2$, and
  $\alpha_n=2\alpha_{n-1}+\alpha_{n-2}$ for $n\ge 3$.  Then, for
  $n\ge2$, the dimension of the degree-$n$ component of $\LLgs$ is the
  number of ways of writing $n-1$ as a sum
  $k_1\alpha_1+\cdots+k_t\alpha_t$ with all $k_i\in\{0,1,2\}$.
\end{prop}

The Gupta-Sidki Lie algebra generalizes to arbitrary characteristic
$p$, with now $\psi(t)=x\otimes a+x^{p-1}\otimes t$.  We will explore
in~\S\ref{ss:gsgp} the connections between $\LLgs$ and the Gupta-Sidki
group.

\subsection{The Grigorchuk Lie algebra}\label{ss:grla}
Again inspired by the self-similarity structure~\eqref{eq:ggp}, we
consider $X=\Ft[x]/(x^2)$, a Lie algebra $\LLg$, and a
restricted Lie algebra $\LLtg$. Both are generated by $a,b,c,d$, and
have the same self-similarity structure (denoting both algebras by $\LL$)
\begin{equation}\label{eq:grla}
  \psi:\begin{cases}
  \LL\to\LL\wr\Der X\\
  a\mapsto\partial_x\\
  b\mapsto x\otimes(a+c)\\
  c\mapsto x\otimes(a+d)\\
  d\mapsto x\otimes b.
\end{cases}
\end{equation}
We seek a grading for $\LL$ that makes the generators
homogeneous. Again $X$ is $\Z$-graded, with $\deg(x)=-1$ so
$\deg(a)=1$, and $\deg(b)=\deg(c)=\deg(d)=\deg(a)=1$, so
$\lambda=2$. In other words, the Lie algebra $\LL$ is no more than
$\Z$-graded.  Using our matrix notation:
\[a\mapsto\begin{pmatrix} 0 & 1\\ 0 & 0\end{pmatrix},\quad
b\mapsto\begin{pmatrix} 0 & 0\\a+c & 0\end{pmatrix},\quad
c\mapsto\begin{pmatrix} 0 & 0\\a+d & 0\end{pmatrix},\quad
d\mapsto\begin{pmatrix} 0 & 0\\b & 0\end{pmatrix}.\]

Contrary to~\S\ref{ss:gsla}, there are important differences between
the algebras $\LLg$ and $\LLtg$. Their relationship is as follows:
$\LLtg$ is an extension of $\LLg$ by the abelian algebra
$\Bbbk\{1^{\otimes n}\otimes [a,b]^2\}$. In this~\S, we denote $\LLg$
by $\LL$.

We note that $\LL$ is not recurrent. However, let us define$
e=a+c,f=a+d$, and set $\LL'=\langle b,e,f,f^2\rangle$.
\begin{prop}
  $\LL'$ is an ideal of codimension $1$ in $\LL$. It is recurrent,
  transitive, and regularly branched on its ideal
  $\KK=\langle[b,e]\rangle$ of codimension $3$.
\end{prop}
\begin{proof}
  First, $\LL'$ is recurrent: indeed $(\varepsilon\otimes1)\psi f^2=b$
  and $(\varepsilon\otimes1)\psi[b,e]=e$; note the relation $b+e+f=0$,
  so $\LL'$ is $2$-generated. By Lemma~\ref{lem:transitive}, $\LL'$ is
  transitive.

  To show that $\LL'$ is branched on $\KK$, it suffices to note
  $\psi[[b,e],b]=x\otimes[b,e]$, so $\psi(\KK)$ contains
  $X\otimes\KK$. Finally, $\LL'/\KK$ has basis $\{b,e,e^2\}$, as a
  direct calculation shows.
\end{proof}

We note, however, that the corresponding subalgebra ${}_2\LL'$ is not
regularly branched on the restricted ideal
${}_2\KK=\langle[b,e]\rangle$; indeed, as we noted above, ${}_2\KK$
contains $1^{\otimes n}\otimes[b,e]^2$ for all $n\in\N$, yet does not
contain $x\otimes[b,e]^2$.

\begin{thm}
  The Lie algebra $\LLtg$ is nil, of unbounded nillity, but not
  nilpotent.
\end{thm}
\begin{proof}
  The ideal $\langle b,c,d\rangle$ in $\LLtg$ has codimension $1$; and
  it is generated by $b,c,[b,a],[c,a]$. Each of these elements is
  bounded and $3$-evanescent; Corollary~\ref{cor:nil} applies. Then
  $\LLtg$ itself is nil, because $a^2=0$.

  That the nillity in unbounded follows from
  Proposition~\ref{prop:unboundednil}. Clearly $\LLtg$ is not
  nilpotent, since by Proposition~\ref{prop:notPI} it is not even PI.
\end{proof}

\begin{prop}
  The relative Hausdorff dimension of $\LL$ and $\LLtg$, with
  $\mathscr P=\Bbbk\partial_x$, is
  \[\Hdim_{\mathscr P}(\LL)=\frac12.\]
\end{prop}
\begin{proof}
  We follow Proposition~\ref{prop:hdimrat}. We may take $M=3$, and
  readily compute $\dim(\LL/\KK)=4$ while $\dim(\KK/(X\otimes\KK))=1$,
  the latter having basis $\{[a,b]\}$. Letting $\LL_n$ denote the
  image of $\LL$ in $\Der(X^{\otimes n})$, we find
  $\dim(\LL_3)=7$. This gives
  \[\dim(\LL_n)=2^{n-1}+3,\]
  and the claimed result.
\end{proof}
Note, on the other hand, that $\dim((\LLtg)_n)=2^{n-1}+n$, by the same
calculation but taking into account the $1^{\otimes n}\otimes[a,b]^2$.

\begin{lem}
  The algebra $\LLg$ is contracting.
\end{lem}
\begin{proof}
  Its nucleus is $\Bbbk\{a,b,d\}$.
\end{proof}

\begin{cor}
  The Gelfand-Kirillov dimension of $\LLg$ and $\LLtg$ is $1$.
\end{cor}
\begin{proof}
  This follows immediately from Propositions~\ref{prop:GKupper}
  and~\ref{prop:GKlower}.
\end{proof}

In fact (see also Theorem~\ref{thm:isogg}), a much stronger result
holds; namely, $\LLg$ and $\LLtg$ have bounded width:
\begin{prop}
  A basis of $\LLg$ is
  \[\{a,b,d,[a,d],x^{i_1}\otimes\cdots\otimes x^{i_n}\otimes[a,b]\mid
  n\in\N,i_k\in\{0,1\}\}.\] The elements $a,b,d$ have degree $1$, the
  element $[a,d]$ has degree $2$, and the element
  $x^{i_1}\otimes\cdots\otimes x^{i_n}\otimes[a,b]$ has degree
  $2^{n+1}-\sum 2^{k-1}i_k$.

  A basis of $\LLtg$ consists of the above basis, with in addition
  elements $1^{\otimes n}\otimes[a,b]^2$ of degree $2^{n+2}$.

  In particular, in $\LLg$ there is a one-dimensional subspace of
  degree $n$ for all $n\ge3$, while in $\LLtg$ there is a
  two-dimensional subspace of degree $n$ for all $n\ge2$ a power of
  two.
\end{prop}
We note that ${}_2\LL'$, being a $2$-generated restricted Lie algebra,
cannot have maximal class, because it is infinite
dimensional~\cite{riley:rlamc}. It could actually be that the growth
of ${}_2\LL'$ is minimal among infinite restricted Lie algebras over a
field of characteristic $2$.

We will explore in~\S\ref{ss:ggp} the connections between $\LLg$,
$\LLtg$ and the Grigorchuk group.

\subsection{Grigorchuk Lie algebras}\label{ss:gla}
We generalize the previous example $\LL'$ as follows. We fix a field
$\Bbbk$ of characteristic $p$, the alphabet $X=\Bbbk[x]/(x^p)$, and an
infinite sequence $\omega=\omega_0\omega_1\dots\in\mathbb
P^1(\Bbbk)^\infty$. Choose a projective lift $\mathbb
P^1(\Bbbk)\to\Hom(\Bbbk^2,\Bbbk)$, and apply it to $\omega$. Consider
also the shift map $\sigma:\omega_0\omega_1\dots\mapsto\omega_1\dots$.

Define then a Lie algebra $\LL_\omega$ acting on $R(X)$, generated by
$\Bbbk^2$, with (non-self!-)similarity structure
\[\psi:\begin{cases}
  \LL_\omega\to\LL_{\sigma\omega}\wr\Der X\\
  \Bbbk^2\ni a\mapsto x^{p-1}\otimes a+\omega_0(a)\partial_x.
\end{cases}\] It is a $\Z$-graded algebra, with dilation $\lambda=p$
and $\deg(a)=1$ for all $a\in\Bbbk^2$. To see better the connection to
the Grigorchuk example, consider $\Bbbk=\Ft$ and
$\omega=(\omega_0\omega_1\omega_2)^\infty$ with
$\omega_0,\omega_1,\omega_2$ the three non-trivial maps
$\Bbbk^2\to\Bbbk$. The three non-trivial elements of $\Bbbk^2$ are
$b,e,f$ generating the subalgebra $\LL'$ from the previous~\S.  Using
our matrix notation,
\[b\mapsto\begin{pmatrix} 0 & 0\\ e & 0\end{pmatrix},\quad
e\mapsto\begin{pmatrix} 0 & 1\\f & 0\end{pmatrix},\quad
f\mapsto\begin{pmatrix} 0 & 1\\b & 0\end{pmatrix}.\]

We summarize the findings of the previous~\S in this more general
context. Recall from~\cite{caranti-m-n:glamc1} that a graded algebra
$\LL$ is of \emph{maximal class} if it is generated by the
two-dimensional subspace $\LL_1$ and $\dim(\LL_n)=1$ for all
$n\ge2$. In particular, it has Gelfand-Kirillov dimension at most $1$.

\begin{prop}
  The algebra $\LL_\omega$ is branched and of maximal class.\qedhere
\end{prop}

The construction of $\LL_\omega$ is modelled on that of the Grigorchuk
groups $G_\omega$ introduced in~\cite{grigorchuk:pgps}. We detail now
the connection to the algebras of maximal class studied by Caranti
\emph{et al}, and recall their definition of
\emph{inflation}~\cite{caranti-m-n:glamc1}*{\S6}. Let $\LL$ be a Lie
algebra of maximal class, and let $\mathscr M$ be a codimension-$1$
ideal, specified by $\omega\in\mathbb P^1(\Bbbk)$: $a\in\LL$ belongs
to $\mathscr M$ if and only if $\omega(a_1)=0$, where $a_1$ denotes
the degree-$1$ part of $a$.

Choose $s\in\LL\setminus\mathscr M$; then $s$ acts as a derivation on
$\mathscr M$. The corresponding \emph{inflated} algebra is
${}^\LL\!\!\!\mathscr M=\mathscr
M\otimes\Bbbk[\varepsilon]/(\varepsilon^p)\rtimes\Bbbk$, where $\Bbbk$
acts by the derivation
$s'=1\otimes\partial_\varepsilon-s\otimes\varepsilon^{p-1}$. Note
$(s')^p=s$. They show that ${}^\LL\!\!\!\mathscr M$ is again an algebra of
maximal class. The main results
of~\cites{caranti-n:glamc2,jurman:glamc3} are that every
infinite-dimensional Lie algebra of maximal class is obtained through
a (possibly infinite) number of steps from an elementary building
blocks such as the Albert-Franks algebras.

In fact, ${}^\LL\!\!\!\mathscr M$ may also be described as the subalgebra of
$\LL\wr\Bbbk\partial_\varepsilon$ generated by $\mathscr M$ and
$s'$. The algebra ${}^\LL\!\!\!\mathscr M$ is independent (up to
isomorphism) of the choice of $s$, and therefore solely depends on the
choice of $\omega$.

The algebras $\LL_\omega$ presented in this~\S\ are examples of Lie
algebras of maximal class that fall into the ``infinitely iterated
inflations'' subclass~\cite{caranti-m-n:glamc1}*{\S9}. However, they
do not appear here as inverse limits, but rather as
countable-dimensional vector spaces, dense in the algebras constructed
by Caranti \emph{et al}.

If $\LL'$ was obtained from $\LL$ through inflation, then $\LL$ may be
recovered from $\LL'$ through \emph{deflation}: choose $s'\in\LL'$ of
degree $1$ that does not commute with $\LL'_2$, and set
$(\LL')^\downarrow=\Bbbk(s')^p\oplus\bigoplus_{n\ge1}(\LL')_{pn}$. Then
$\LL\cong(\LL')^\downarrow$.

\subsection{Fabrykowski-Gupta Lie algebras}
Again inspired by the self-similarity structure of the
Fabrykowski-Gupta group~\cite{fabrykowski-g:growth1}, we consider
$X=\Fp[x]/(x^p)$ and a Lie algebra $\LL$ generated by $a,t$
with self-similarity structure
\[\psi:\begin{cases}
  \LL\to\LL\wr\Der X\\
  a\mapsto\partial_x\\
  t\mapsto x^{p-1}\otimes(a+t).
\end{cases}\] We seek a grading for $\LL$ that makes the generators
homogeneous. Again $X$ is $\Z$-graded, with $\deg(x)=-1$ so
$\deg(a)=1$, and $\deg(t)=\deg(a)=1$, while $\lambda\deg(t)=p-1+\deg(t)$ so
$\lambda=p$. In other words, the Lie algebra $\LL$ is no more than $\Z$-graded.
Using our matrix notation:
\[a\mapsto\begin{pmatrix}
  0 & 1 & \cdots & 0\\
  \vdots & \ddots & \ddots & \vdots\\
  \vdots && \ddots & 1\\
  0 &\cdots &\cdots& 0\end{pmatrix},\quad
t\mapsto\begin{pmatrix}
  0 & 0 & \cdots & 0\\
  0 & \ddots & \ddots & \vdots\\
  \vdots & \ddots & \ddots & 0\\
  -a-t & \cdots & 0 & 0\end{pmatrix}.\]

\begin{thm}
  The Lie algebra $\LL$ is not nil.
\end{thm}
\begin{proof}
  Consider the element $x=a+t$. A direct calculation gives
  $\psi(x^p)=-1\otimes x$. It follows that, if $x^{p^s}=0$ for some
  $s>0$, then $-1\otimes x^{p^{s-1}}=0$ so $x^{p^{s-1}}=0$. Since
  $x\neq0$, it is not nil.
\end{proof}

\subsection{Petrogradsky-Shestakov-Zelmanov algebras}\label{ss:pszla}
We consider $\Bbbk$ a field of characteristic $p$, and
$X=\Bbbk[x]/(x^p)$. We fix an integer $m\ge2$, and consider the Lie
algebra $\LL_{m,\Bbbk}$ with generators $d_1,\dots,d_{m-1},v$ and
self-similarity structure
\[\psi:\begin{cases}
  \LL_{m,\Bbbk}\to\LL_{m,\Bbbk}\wr\Der X\\
  d_1\mapsto\partial_x\\
  d_{n+1}\mapsto 1\otimes d_n\text{ for }n=1,\dots,m-2,\\
  v\mapsto 1\otimes d_{m-1}+x^{p-1}\otimes v.
\end{cases}\] In the special case $m=2,\Bbbk=\Ft$,
Petrogradsky actually considered in~\cite{petrogradsky:sila} the
subalgebra $\langle v,[d_1,v]\rangle$, of codimension $1$ in
$\LL_{2,\Ft}$. Shestakov and Zelmanov consider
in~\cite{shestakov-z:nillie} the case $m=2$; see
also~\cite{petrogradsky-s:sila}. We repeat for clarity that last
example (with $d=d_1$) using our matrix notation. For that purpose, we
take divided powers $\{1,x,x^2/2,\dots,x^{p-1}/(p-1)!\}$ as basis of
$X$. The endomorphisms $m_x$ and $m_{\partial_x}$ are
respectively
\[m_x=\begin{pmatrix}
  0 & 0 & \cdots & 0\\
  1 & \ddots & \ddots & \vdots\\
  \vdots & 2 & \ddots & 0\\
  0 &\cdots & p-1 & 0\end{pmatrix},\quad
 m_{\partial_x}=\begin{pmatrix}
  0 & 1 & \cdots & 0\\
  \vdots & \ddots & \ddots & \vdots\\
  \vdots && \ddots & 1\\
  0 &\cdots &\cdots& 0\end{pmatrix}.\]

It follows that the matrix decomposition of $d,v$ are
\[d\mapsto\begin{pmatrix}
  0 & 1 & \cdots & 0\\
  \vdots & \ddots & \ddots & \vdots\\
  \vdots && \ddots & 1\\
  0 &\cdots &\cdots& 0\end{pmatrix},\quad
v\mapsto\begin{pmatrix}
  d & 0 & \cdots & 0\\
  0 & \ddots & \ddots & \vdots\\
  \vdots & \ddots & \ddots & 0\\
  -v & \cdots & 0 & d\end{pmatrix}.\]

We seek a grading that makes the generators homogeneous. The ring $X$ is $\Z$-graded by setting $\deg(x)=-1$, so $T(X)$ is
$\Z[\lambda]$-graded. If that grading is to be compatible with the
self-similarity structure, however, we must impose
$\deg(d_n)=\lambda^{n-1}$ and
\[\deg(v)=\lambda\deg(d_{m-1})=-(p-1)+\lambda\deg(v)=\lambda^{m-1},\]
so $\lambda^m-\lambda^{m-1}+1-p=0$. We therefore grade $T(X)$ and
$\LL_{m,\Bbbk}$ by the abelian group
\[\Lambda=\Z[\lambda]/(\lambda^m-\lambda^{m-1}+1-p).\]

For simplicity, we state the following result only in the case $m=2$:
\begin{prop}
  The Lie algebra $\LL_{2,\Bbbk}$ is regularly branched on its ideal
  $\langle[v,[d,v]]\rangle$ of codimension $p+1$.
\end{prop}
\begin{proof}
  First, $\LL_{2,\Bbbk}$ is recurrent: indeed
  $(\varepsilon\otimes1)\psi(v)=d$ and
  $(\varepsilon\otimes1)\psi[[d^{p-1},v]]=v$. Then, by
  Lemma~\ref{lem:transitive}, $\LL_{2,\Bbbk}$ is transitive.

  Write $c=[v,[d,v]]$ and $\KK=\langle c\rangle$. To prove that $\LL_{2,\Bbbk}$
  is branched on $\KK$, it suffices to exhibit $c'\in\KK$ with
  $\psi(c')=x^{p-1}\otimes c$. If $p=2$, then $c'=[v,c]$ will do,
  while if $p\ge3$ then $c'=[[d,v],[[d^{p-3},c]]]$ will do.

  A direct computation shows that $\LL_{2,\Bbbk}/\KK$ is finite-dimensional,
  with basis $\{d,v,[d,v],\cdots,[[d^{p-1},v]]\}$.
\end{proof}

We are now ready to reprove the following main result from Shestakov
and Zelmanov:
\begin{thm}[\cite{shestakov-z:nillie}*{Example~1}]
  The algebra $\LL_{m,\Bbbk}$ is nil but not nilpotent.
\end{thm}
\begin{proof}
  The ideal $\langle v\rangle$ has finite codimension and is generated
  by $m$-evanescent elements, so Corollary~\ref{cor:nil} applies to
  it. We conclude by noting that $\LL_{m,\Bbbk}/\langle v\rangle$ is abelian and
  hence nil.
\end{proof}

\noindent We concentrate again on $m=2$ in the following results:
\begin{prop}
  The relative Hausdorff dimension of $\LL_{2,\Bbbk}$, with $\mathscr
  P=\Bbbk\partial_x$, is
  \[\Hdim_{\mathscr P}(\LL_{2,\Bbbk})=\frac{(p-1)^2}{p^3}.\]
\end{prop}
\begin{proof}
  We follow Proposition~\ref{prop:hdimrat}, using the notation
  $\LL=\LL_{2,\Bbbk}$. We may take $M=3$, and readily compute
  $\dim(\LL/\KK)=p+1$ with basis $\{d,[[d^i,v]]\}$ and
  $\dim(\KK/(X\otimes\KK))=(p-1)^2$ with basis $\{[[d^i,v^j,c]]\mid
  i,j\in\{0,\dots,p-2\}$. Letting $\LL_n$ denote the image of $\LL$ in
  $\Der(X^{\otimes n})$, we find $\dim(\LL_3)=p+1$. This gives
  \[\dim(\LL_n)=(p-1)\cdot p^{n-3}+2,\]
  and the claimed result.
\end{proof}

\begin{lem}
  The algebra $\LL_{m,\Bbbk}$ is contracting.
\end{lem}
\begin{proof}
  Its nucleus is $\Bbbk\{d_1,\dots,d_{m-1},v\}$.
\end{proof}

\begin{cor}
  The Gelfand-Kirillov dimension of $\LL_{m,\Bbbk}$ is $\log p/\log\lambda$.
\end{cor}
\begin{proof}
  This follows immediately from Propositions~\ref{prop:GKupper} and~\ref{prop:GKlower}.
\end{proof}

Note, however, that careful combinatorics give a much sharper
result. For example, again for $m=p=2$, the dimension of the span of
commutators of length $\le n$ in $\LL_{2,\Ft}$ is, for
$F_{k-1}<n\le F_k$ and $(F_k)$ the Fibonacci numbers, the number of
manners of writing $F_k-n$ as sum of distinct Fibonacci numbers among
$F_1,\dots,F_{k-4}$. In particular, for $n=F_k$ at least $5$, there is
precisely one commutator of length $n$, namely $1^{\otimes k-2}\otimes
v$.
%element at $F_n+1$ is $t_0\dots t_n v_{n+?}$.

%%%%%%%%%%%%%%%%%%%%%%%%%%%%%%%%%%%%%%%%%%%%%%%%%%%%%%%%%%%%%%%%
\section{Self-similar associative algebras}\label{ss:assoc}
At least three associative algebras may be associated with a
self-similar Lie algebra $\LL$. The first is the universal enveloping
algebra $\U(\LL)$, which maps onto the other two. The second is the
adjoint algebra of $\LL$, namely the associative subalgebra
$\mathfrak{Adj}(\LL)$ of $\End(\LL)$ generated by the derivations
$[a,-]:\LL\to\LL$ for all $a\in\LL$. The third one is the
\emph{thinned algebra} $\AA(\LL)$, defined as follows.

Let us write $d=\dim X$. Recall from~\eqref{eq:matext} that the
self-similarity structure $\psi:\LL\to\LL\wr\Der X$ gave
rise to a linear map $\psi':\LL\to\Mat_d(\LL\oplus\Bbbk)$. We extend
this map multiplicatively to an algebra homomorphism
$\psi':T(\LL)\to\Mat_d(T(\LL))$. Now, for $a=\sum x_i\otimes a_i+\delta$
and $b=\sum y_j\otimes b_j+\epsilon$ in $\LL$, we have $[a,b]=\sum\sum
x_iy_j\otimes[a_i,b_j]+\sum\delta y_j\otimes b_j-\sum\epsilon
x_i\otimes a_i+[\delta,\epsilon]$, and therefore
\begin{align*}
  \psi'(ab - ba - [a,b])&=\sum\sum
  m_{x_iy_j}(a_ib_j-b_ja_i-[a_i,b_j])\\
  &+\sum(m_{x_i}m_\epsilon-m_\epsilon m_{x_i}-m_{\epsilon x_i})a_i\\
  &-\sum(m_{y_j}m_\delta-m_\delta m_{y_j}-m_{\delta y_j})b_j\\
  &+m_\delta m_\epsilon-m_\epsilon m_\delta-m_{[\delta,\epsilon]};
\end{align*}
the last three summands are zero, so all entries of
$\psi'(ab-ba-[a,b])$ lie in the ideal generated by the
$a'b'-b'a'-[a',b']$. We deduce:
\begin{prop}
  The map $\psi'$ induces an algebra homomorphism $\psi':\U(\LL)\to
  \Mat_d(\U(\LL))$.
\end{prop}

Now we note that, even though $\psi:\LL\to\LL\wr\Der X$ may be
injective, this does not imply that $\psi'$ is injective. As a simple
example, consider the Grigorchuk Lie algebra from~\S\ref{ss:grla}. We
have $bc\neq0$ in $\U(\LL)$, but $\psi'(bc)=0$.

Since $\LL$ acts on $T(X)$ by $\Bbbk$-linear maps, the universal
enveloping algebra $\U(\LL)$ also acts on $T(X)$. Quite clearly, the
kernel of the action of $\U(\LL)$ on $T(X)$ contains the kernel of
$\psi'$.

\begin{defn}\label{defn:thinned}
  Let $\LL$ be a self-similar Lie algebra. The \emph{thinned
    enveloping algebra} of $\LL$ is the quotient $\AA(\LL)$ of
  $\U(\LL)$ by the kernel of its natural action on $T(X)$.
\end{defn}

Rephrasing the results from~\S\ref{ss:examples}, we deduce that, for
$\LL=$ the Gupta-Sidki Lie algebra, the Grigorchuk Lie algebra, and
the Petrogradsky-Shestakov-Zelmanov algebras, the image of $\LL$ in
$\U(\LL)$ is nil; the same property then holds for the adjoint algebra
and the thinned algebra of $\LL$. Apart from this information, the
structure of $\U(\LL)$ or $\mathfrak{Adj}(\LL)$ seems mysterious.

In case $\LL$ admits two gradings with different dilations, it is
possible to deduce that $\U(\LL)$ is a sum of locally nilpotent
subalgebras. Assume therefore that $\LL$ admits two degree functions,
written $\deg_\lambda$ and $\deg_\mu$. For simplicity, assume also
$\lambda>|\mu|>1$, although the case $|\mu|\le1$ can also be handled
as a limit, or by a small change in the argument; indeed increasing $|\mu|$ only
makes inequalities tighter in what follows.  The following is drawn
from~\cite{petrogradsky-s:sila}*{Theorem~2.1}:
\begin{prop}\label{prop:sum}
  Let $\LL$ be a graded self-similar Lie algebra, in characteristic
  $p>0$, with gradings $\deg_\lambda,\deg_\mu$, such that $\LL$ is
  generated by finitely many positive-degree elements with respect to
  $\deg_\lambda$. Then there is a decomposition as a sum of
  subalgebras
  \[\U(\LL)=\U_+\oplus\U_0\oplus\U_-,\]
  in which $\U_+,\U_0,\U_-$ are respectively the spans of homogeneous
  elements $a$ with $\deg_\mu(a)>0$, respectively $=0$, respectively
  $<0$; and $\U_+$ and $\U_-$ are locally nilpotent.

  In particular, homogeneous elements with $\deg_\mu\neq0$ are nil.
\end{prop}
\begin{proof}
  We first deduce from~\eqref{eq:contrub} and~\eqref{eq:contrlb} that,
  for appropriate constants $C,D$, we have for all $a\in X^{\otimes
    n}\otimes N$
  \[\log_{|\mu|}(\deg_\mu(a)/C)\le n\le\log_\lambda(\deg_\lambda(a)/D).\]
  Setting $\theta=\log|\mu|/\log\lambda\in(0,1)$, we get, for a fresh constant $C$,
  \[\deg_\mu(a)\le C\deg_\lambda(a)^\theta.\]
  We seek a similar inequality for $\U(\LL)$. For that purpose, choose
  a homogeneous basis $(a_1,a_2,\dots)$ of $\LL$, and recall that
  $\U(\LL)$ has basis
  \[\big(\prod_{i\ge1}a_i^{n_i}\mid n_i\in\Fp,\text{ almost all }0\big).\]
  Each $a_i$ belongs to $X^{\otimes m_i}\otimes N$ for some minimal
  $m_i\in\N$. For $j\in\N$, let $\ell_j$ denote the number of $m_i$
  with $m_i=j$. Then $\ell_j\le(\dim X)^j\dim N$.

  Consider now $u=\prod_{i\ge1}a_i^{n_i}\in\U(\LL)$. We have
  \[\deg_\lambda(u)=\sum_{i\ge1}n_i\lambda^{m_i},\quad|\deg_\mu(u)|\le\sum_{i\ge1}n_i|\mu|^{m_i},\]
  and we wish to study cases in which $|\deg_\mu(u)|$ is as large as
  possible for given $\deg_\lambda(u)$. Because $|\mu|<\lambda$ and by
  convexity of the functions $\lambda^x,|\mu|^x$, this will occur when
  $n_i=p-1$ whenever $m_i$ is small, say $m_i<K$, and $n_i=0$ whenever
  $m_i>K$; with intermediate values whenever $m_i=K$. Letting
  $A\propto B$ mean that the ration $A/B$ is universally bounded away
  from $0$ and $\infty$, we have
  \begin{align}
    \deg_\lambda(u)&\propto(p-1)\sum_{j=0}^K\ell_j\lambda^j\propto(\lambda\dim X)^K,\notag\\
    \deg_\mu(u)&\overset<\propto(p-1)\sum_{j=0}^K\ell_j|\mu|^j\propto(|\mu|\dim X)^K,\notag\\
    \intertext{so}
    \deg_\mu(u)&<C\deg_\lambda(u)^{\theta'},\label{eq:theta'}
  \end{align}
  for a constant $C$ and $\theta'=\log(|\mu|\dim X)/\log(\lambda\dim X)\in(0,1)$.

  It is clear that we have a decomposition
  $\U(\LL)=\U_+\oplus\U_0\oplus\U_-$. We now show that $\U_+$ is
  locally nilpotent, the same argument applying to $\U_-$.  For that
  purpose, consider $u_1,\dots,u_k\in\U_+$, spanning a subspace $V$,
  and set
  \[D=\max_{i\in\{1,\dots,k\}}\frac{\deg_\lambda(u_i)}{\deg_\mu(u_i)}.\]
  For $s\in\N$, consider a non-zero homogeneous element $u$ in the
  $s$-fold product $V^s$. Then $\deg_\lambda(u)\le D\deg_\mu(u)$;
  combining this with~\eqref{eq:theta'} we get
  \[\deg_\mu(u)\le C(D\deg_\mu(u))^{\theta'},\]
  so $\deg_\mu(u)$ is bounded and therefore $s$ is also bounded.
\end{proof}
Note that, if $\LL$ is generated by a set $S$ such that the
$\deg_\mu(s)$ with $s\in S$ are linearly independent in $\R$, then
$\U_0=\Bbbk$ and therefore homogeneous elements except scalars are nil
in $\U(\LL)$.

On the other hand, note that even the quotient algebra $\AA(\LL)$
tends to have transcendental elements. For example, $\AA(\LLtg)$
contains $a+b+ad$, which is transcendental
by~\cite{bartholdi:branchalgebras}*{Theorem~4.20} and
Theorem~\ref{thm:isothinnedg}. It seems that the element $a^2+t$ of
$\AA(\LLgs)$ is transcendental.
\begin{lem}
  If $p=\alpha^m-\alpha^{m-1}+1$ for some $\alpha\in\N$, then
  $\AA(\LL_{m,\Fp})$ is not nil.
\end{lem}
\begin{proof}
  Write $\alpha=\beta-1$, and consider the element
  $x=d_1^{\alpha\beta^{m-2}}d_2^{\alpha\beta^{m-3}}\cdots
  d_{m-1}^\alpha v$. Then $\psi(x^\beta)$ is a lower triangular
  matrix, with $x$ at position $(p,p)$. Because $x\neq0$, it follows
  that $x^{\beta^n}\neq0$ for all $n\in\N$, so $x$ is not a nil
  element.
\end{proof}
There does not seem to exist such a simple argument for arbitrary
primes; for instance, for $p=m=2$ it seems that $x=v+v^2+vuv$ has
infinite order (its order is at least $2^{10}$), while for $p=5$ and
$m=2$ it seems that $x=v+d^2-d^4$ has infinite order (its order is at
least $5^3$).

\subsection{Thinned algebras}
The ``thinned algebra'' from Definition~\ref{defn:thinned} is an
instance of a \emph{self-similar algebra}, as defined
in~\cite{bartholdi:branchalgebras}. We recall the basic notion: a
\emph{self-similar associative algebra} is an algebra $\AA$ endowed
with a homomorphism $\phi:\AA\to \Mat_d(\AA)$. If furthermore $\AA$ is
augmented (by $\varepsilon:\AA\to\Bbbk$), then $\AA$ acts linearly on
$T(\Bbbk^d)$: given $a\in\AA$ and an elementary tensor
$v=x_1\otimes\cdots\otimes x_n\in T(\Bbbk^d)$, set
$a\cdot1=\varepsilon(a)$ and recursively
\[a(v)=\sum_{i,j=1}^d\langle e_i|x_1\rangle e_j\otimes
a_{ij}(x_2\otimes\cdots\otimes x_n),\] if
$\psi(a)=(a_{ij})$. Conversely, specifying $\phi(s),\varepsilon(s)$
for generators $s$ of $\AA$ defines at most one self-similar algebra
acting faithfully on $T(\Bbbk^d)$.

The first example of self-similar algebra (though not couched in that
language) is due to Sidki~\cite{sidki:primitive}. He constructed a
primitive ring $\AA$, containing both the Gupta-Sidki torsion group
(see~\S\ref{ss:gsgp}) and a transcendental element.

Another example~\cite{bartholdi:branchalgebras}, on the other hand, is
far from primitive: it equals its Jacobson radical, contains both the
Grigorchuk group (see~\S\ref{ss:ggp}) and an invertible transcendental
element, and has quadratic growth.

The fundamental idea of ``linearising'' the definition of a
self-similar group (see~\S\ref{ss:ssg}) already appears
in~\cite{zarovnyi:linear}.

\subsection{Bimodules}
The definition of self-similar associative algebra, given above, has
the defect of imposing a specific choice of basis. The following more
abstract definition is essentially equivalent.
\begin{defn}\label{defn:ssaa}
  An associative algebra $\AA$ is \emph{self-similar} if it is
  endowed with a \emph{covering bimodule}, namely, an
  $\AA$-$\AA$-bimodule $\mathscr M$ that is free qua right
  $\AA$-module.
\end{defn}
Indeed, given $\psi:\AA\to\Mat_d(\AA)$, define $\mathscr
M=\Bbbk^d\otimes\AA$, with natural right action, and left action
\[a\cdot(e_i\otimes b)=\sum_{j=1}^d e_j\otimes a_{ij}b\text{ for
}\psi(a)=(a_{ij}).\] Conversely, if $\mathscr M$ is free, choose an
isomorphism $\mathscr M_\AA\cong X\otimes\AA$ for a vector space $X$,
and choose a basis $(e_i)$ of $X$; then write $\psi(a)=(a_{ij})$ where
$a\cdot(e_i\otimes1)=\sum e_j\otimes a_{ij}$ for all $i$.

Note that we had no reason to require $X$ to be finite-dimensional (of
dimension $d$) in Definition~\ref{defn:ssaa}, though all our examples
are of that form.

The natural action of $\AA$ may be defined without explicit reference
to a basis $X$ of $\mathscr M_\AA$: one simply lets $\AA$ act on the
left on
\[\bigoplus_{n\ge0}\mathscr M\otimes_\AA\cdots\otimes_\AA\mathscr M\otimes_\AA\Bbbk.\]

It is possible to define self-similar Lie algebras $\LL$ in a
basis-free manner, by requiring that the universal enveloping algebra
$\U(\LL)$ be endowed with a covering bimodule. We shall not follow
that approach here, because it is not (yet) justified by any
applications.

\subsection{Growth}
Under the assumptions of Propositions~\ref{prop:GKupper}
and~\ref{prop:GKlower}, we may estimate the growth of the associative
algebras $\U(\LL)$ and $\AA(\LL)$. We begin by $\U(\LL)$, for which
analytic number theory methods are useful. The following is an
adaptation of~\cite{nathanson:density}*{Theorem~1}. Note that, with a
little more care, Nathanson obtained the same bounds $\Phi_1=\Phi_2$,
for $\alpha=1/2$.
\begin{lem}\label{lem:nathanson}
  Let $f(x)=\sum_{n\ge0}a_nx^n$ be a power series with positive
  co\"efficients, and consider $\alpha\in(0,1)$.
  \begin{enumerate}
  \item There exists a homeomorphism $\Phi_1:[0,\infty]\to[0,\infty]$ such that
    \[\liminf_{n\to\infty}\frac{\log a_n}{n^\alpha}\ge L\text{ implies }\liminf_{x\to1^-}(1-x)^{\alpha/(1-\alpha)}\log f(x)\ge\Phi_1(L).\]
  \item There exists a homeomorphism $\Phi_2:[0,\infty]\to[0,\infty]$ such that
    \[\limsup_{n\to\infty}\frac{\log a_n}{n^\alpha}\le L\text{ implies }\limsup_{x\to1^-}(1-x)^{\alpha/(1-\alpha)}\log f(x)\le\Phi_2(L).\]
  \end{enumerate}
\end{lem}
\begin{proof}
  Ad~(1): for every $\epsilon>0$ there are arbitrarily large $n\in\N$
  with $a_n\ge e^{(L-\epsilon)n^\alpha}$. Consider $x=e^{-t}$, with
  $t\in\R_+$. Then $f(x)\ge a_nx^{-n}\ge
  e^{(L-\epsilon)n^\alpha-tn}$. This expression is maximized at
  $t=\alpha(L-\epsilon)n^{\alpha-1}$, with $t\to0$ as
  $n\to\infty$. Therefore, $\log
  f(x)\ge(1-\alpha)(L-\epsilon)(t/\alpha(L-\epsilon))^{\alpha/(\alpha-1)}=Kt^{\alpha/(\alpha-1)}$
  for a function $K(L,\epsilon)$. Now, for $t\to0$, we have
  $t\approx1-x$, so $(1-x)^{\alpha/(1-\alpha)}\log f(x)\ge K$ for $x$
  near $1^-$. Thus $\Phi_1(L):=\lim_{\epsilon\to0}K(L,\epsilon)$
  satisfies~(1).

  Ad~(2): for every $\epsilon>0$ there is $N_0\in\N$ such that $a_n\le
  e^{(L-\epsilon)n^\alpha}$ for all $n\ge N_0$. Consider $x$ near
  $1^-$, and write again $x=e^{-t}$. Set
  $N_1=(t/\alpha(L+\epsilon))^{1/(\alpha-1)}$. Write
  \[f(x)=\sum_{n<N_0}a_nx^n+\sum_{n=N_0}^{2N_1}a_nx^n+\sum_{n\ge
    2N_1}a_nx^n.\] The first summand is bounded by a function
  $K_1(\epsilon)$. The second is bounded by
  $2N_1e^{(L+\epsilon)N_1^\alpha-N_1t}$. For $n\ge2N_1$ we have
  $d/dn((L+\epsilon)n^\alpha-tn)\le(2^{\alpha-1}-1)\alpha(L+\epsilon)N_1^{\alpha-1}<0$,
  so the third summand is bounded by the geometric series
  $\sum_{i\ge0}e^{(L+\epsilon)(2N_1)^\alpha-t(2N_1)}e^{(2^{\alpha-1}-1)\alpha(L+\epsilon)N_1^{\alpha-1}i}$.
  Collecting all three summands into a bound $\Phi_2(L)$ yields~(2).
\end{proof}

\begin{lem}\label{lem:prodUEA}
  Consider the series $f(x)=\prod_{n\ge1}(1-x^n)^{-n^\beta}$. Then
  $(1-x)^{\beta+1}\log f(x)$ is bounded away from $\{0,\infty\}$ as
  $x\to1^-$.
\end{lem}
\begin{proof}
  We have $\log f(x)=-\sum_{n\ge1}n^\beta\log(1-x^n)$; and
  $-\log(1-x^n)\ge x^n$ so $\log f(x)\ge\sum_{n\ge1}n^\beta
  x^n\approx(1-x)^{-\beta-1}$; therefore $L_f:=(1-x)^{\beta+1}\log
  f(x)$ is bounded away from $0$.

  Conversely, it makes no difference to consider $f(x)$ or
  $g(x):=f(x)/f(x^p)$, because $L_g=(1-p^{-1})L_f$. Now
  $-\log((1-x^n)/(1-x^{pn}))=\sum_{i\ge1}(x^{ni}-x^{pni})/i\le px^n$,
  so $\log g(x)\precapprox p(1-x)^{-\beta-1}$ is bounded
  away from $\infty$.
\end{proof}

\noindent The following is an improvement on the bounds given
in~\cite{petrogradsky:growth2}*{Proposition~1}.
\begin{thm}
  Let $\LL$ be a contracting, $\R_+$-graded self-similar Lie algebra,
  with dilation $\lambda>1$, that is generated by finitely many
  positive-degree elements. Then the universal enveloping algebra
  $\U(\LL)$ has subexponential growth; more precisely, the growth of
  $\U(\LL)$ is bounded as
  \begin{equation}\label{eq:growthUEA}
    \dim(\U(\LL))_{\le d}\precsim e^{d^{\GKdim(\LL)/(\GKdim(\LL)+1)}}.
  \end{equation}
  If furthermore $\LL$ is regularly weakly branched, then the exponent
  in~\eqref{eq:growthUEA} is sharp.
\end{thm}
\begin{proof}
  We denote, for an algebra $\AA$, by $\AA_n$ the homogeneous summand
  of degree $n$.  Let $f(x)=\sum_{n\ge0}\dim\U(\LL)_nx^n$ denote the
  Poincar\'e series of $\U(\LL)$. By the Poincar\'e-Birkhoff-Witt
  Theorem, we have $f(x)=f_0(x)=\prod_{n\ge1}(1-x^n)^{-\dim\LL_n}$ in
  characteristic $0$, and $f(x)=f_0(x)/f_0(x^p)$ in characteristic
  $p$. By Proposition~\ref{prop:GKupper}, we have
  \begin{equation}\label{eq:UEA:1}
    \sum_{n=\lambda^m}^{\lambda^{m+1}-1}\dim\LL_n\precapprox\lambda^{\theta
      m},\text{ with }\theta=\GKdim(\LL).
  \end{equation}
  The estimate~\eqref{eq:growthUEA} only depends on the asymptotics of
  $f(x)$ near $1$, and these change only by a factor of $\lambda$
  between the extreme cases when the sum in~\eqref{eq:UEA:1} is
  concentrated in its first or its last terms. It therefore makes no
  harm to assume $\dim\LL_n\precapprox n^{\theta-1}$, so
  $(1-x)^\theta\log(x)$ is bounded away from $\infty$ by
  Lemma~\ref{lem:prodUEA}; and the upper bound in~\eqref{eq:growthUEA}
  follows from Lemma~\ref{lem:nathanson}(1).

  On the other hand, if $\LL$ is regularly weakly branched then by
  Proposition~\ref{prop:GKlower} we have
  $\sum_{n=\lambda^m}^{\lambda^{m+1}-1}\dim\LL_n\succapprox\lambda^{\theta
    m}$ and the lower bound in~\eqref{eq:growthUEA}
  follows from Lemma~\ref{lem:nathanson}(2).
\end{proof}

We now show that the thinned algebra $\AA(\LL)$ has finite
Gelfand-Kirillov dimension, double that of $\LL$, under the same
hypotheses:
\begin{thm}
  Let $\LL$ be a contracting, $\R_+$-graded self-similar Lie algebra,
  with dilation $\lambda>1$, that is generated by finitely many
  positive-degree elements. Then
  \[\GKdim\AA(\LL)\le2\frac{\log\dim X}{\log\lambda}.\]
  On the other hand, if $\LL$ is regularly weakly branched, then
  \[\GKdim\AA(\LL)\ge2\frac{\log\dim X}{\log\lambda}.\]
\end{thm}
\begin{proof}
  Let $\AA_d$ denote the span of homogeneous elements in $\AA(\LL)$ of
  degree $\le d$. Consider $a\in\AA_d$, and express it in
  $\Mat_{X^{\otimes n}\times X^{\otimes n}}(N)$. By~\eqref{eq:contrlb}
  we have $n\le\log_\lambda(d/\epsilon)$, so
  \[\dim\AA_d\le\sum_{j=0}^n(\dim X)^{2j}\dim N\precapprox
  d^{2\log(\dim X)/\log\lambda}.\]

  If $\LL$ is regularly weakly branched then, following the proof of
  Proposition~\ref{prop:GKlower}, we have
  \[\dim\AA_d\ge(d/\epsilon)^{2\log(\dim X)/\log\lambda}.\qedhere\]
\end{proof}

There is yet another notion of growth, which we just mention in
passing. For $v\in R(X)$, consider the \emph{orbit growth} function
\[\gamma_v(d)=\dim(\LL_dv),\]
where $\LL_d$ denotes the span of homogeneous elements in $\AA(\LL)$
of degree $\le d$; and consider also
\[\gamma(d)=\sup_{v\in R(X)}\gamma_v(d).\]
It seems that $\gamma(d)$ is closely related to $\dim\LL_d$, but that
may be an artefact of the simplicity of the examples yet considered.

\subsection{Sidki's monomial algebra}
Sidki considered in~\cite{sidki:monomial} a self-similar associative
algebra $\AA$ on two generators $s,t$, given by the self-similarity
structure
\[s\mapsto\begin{pmatrix}0&0\\1&0\end{pmatrix},\qquad
t\mapsto\begin{pmatrix}0&t\\0&s\end{pmatrix}.\] He gives a presentation
for $\AA$, all of whose relators are monomial, and shows that
monomials are nil in $\AA$ while $s+t$ is transcendental.

The self-similarity structure of $\AA$ is close both to that of
$\AA(\LL_2)$, the difference being the `$s$' at position $(1,1)$ of
$\psi(t)$; and to the thinned ring $\AA(I_4)$ of a semigroup
considered in~\cite{bartholdi-r:i4}, the difference being the
`$1$' at position $(1,2)$ in $\psi(s)$. Note also that $\AA$ and
$\AA(I_4)$ have the same Gelfand-Kirillov dimension. One is led to
wonder whether $\AA$ may be obtained as an associated graded of
$\AA(I_4)$.

%%%%%%%%%%%%%%%%%%%%%%%%%%%%%%%%%%%%%%%%%%%%%%%%%%%%%%%%%%%%%%%%
\section{Self-similar groups}\label{ss:ssg}
My starting point, in studying self-similar Lie algebras, was the
corresponding notion for groups:
\begin{defn}
  Let $X$ be a set, called the \emph{alphabet}. A group $G$ is
  \emph{self-similar} if it is endowed with a homomorphism
  \[\psi:G\to G^X\rtimes\Sym X=:G\wr\Sym X,\]
  called its \emph{self-similarity structure}.
\end{defn}
The first occurrence of this definition seems to
be~\cite{scott:constructfpisg}*{Page~310}; it has also appeared in the
context of groups generated by
automata~\cite{zarovnyi:automaticmappings}.

A self-similar group naturally acts on the set $X^*$ of words over the
alphabet $X$: given $g\in G$ and $v=x_1\dots x_n$, define recursively
\[g(v)=\pi(x_1)g_{x_1}(x_2\dots x_n)\text{ where }\psi(g)=((g_x)_{x\in
  X},\pi).\] Conversely, if $\psi(g)$ is specified for the generators
of a group $G$, this defines at most one self-similar group acting
faithfully on $X^*$.

Note that we have no reason to require $X$ to be finite, though all
our examples are of that form.

Nekrashevych~\cite{nekrashevych:ssg} introduced a more abstract,
essentially equivalent notion:
\begin{defn}
  A group $G$ is \emph{self-similar} if it is endowed with a
  \emph{covering biset}, namely, a $G$-$G$-biset $M$ that is free qua
  right $G$-set.
\end{defn}
Indeed, given $\psi:G\to G\wr\Sym X$, define $M=X\times G$, with
natural right action, and left action
\[g(x,h)=(\pi(x),g_xh)\text{ for }\psi(g)=((g_x)_{x\in X},\pi).\]
Conversely, if $M$ is free, choose an isomorphism $M_G\cong X\times G$
for a set $X$, and write $\psi(g)=((g_x),\pi)$ where
$g(x,1)=(\pi(x),g_x)$ for all $x\in X$.

The natural action of $G$ may be defined without explicit reference
to a ``basis'' $X$ of $M$: one simply lets $G$ act on the left on
\[\bigsqcup_{n\ge0}M\times_G\cdots\times_GM\times_G\{*\},\]
where the fibred product of bisets in $M\times_GN=M\times
N/(mg,n)=(m,gn)$.

If $G$ is a self-similar group, with self-similarity structure
$\psi:G\to G\wr\Sym X$, and $\Bbbk$ is a ring, then its \emph{thinned
  algebra} is a self-similar associative algebra $\AA(G)$ with
alphabet $\Bbbk X$. It is defined as the quotient of the group ring
$\Bbbk G$ acting faithfully on $T(\Bbbk X)$, for the self-similarity
structure $\psi':\Bbbk G\to\Mat_X(\Bbbk G)$ given by
\[\psi'(g)=\sum_{x\in X}{\mathbb1}_{x,\pi(x)}g_x\text{ for
}\psi(g)=((g_x),\pi),\]
where $\mathbb1_{x,y}$ is the elementary matrix with a $1$ at position
$(x,y)$.

The definition is even simpler in terms of bisets and bimodules: if
$G$ has a covering biset $M$, then $\Bbbk G$ has a covering module
$\Bbbk M$, turning it into a self-similar associative algebra.

\subsection{From groups to Lie algebras}
We shall consider two methods of associating a Lie algebra to a
discrete group. The first one, quite general, is due to
Magnus~\cite{magnus:lie}. Given a group $G$, consider its \emph{lower
  central series} $(\gamma_n)_{n\ge1}$, defined by $\gamma_1=G$ and
$\gamma_n=[G,\gamma_{n-1}]$ for $n\ge2$. Form then
\[\LL^\Z(G)=\bigoplus_{n\ge1}\gamma_n/\gamma_{n+1}.\]
This is a graded abelian group; and the bracket
$[g\gamma_{n+1},h\gamma_{m+1}]=[g,h]\gamma_{n+m+1}$, extended
bilinearly, gives it the structure of a graded Lie ring.

Consider now a field $\Bbbk$ of characteristic $p$, and define the
\emph{dimension} series $(\gamma_n^p)_{n\ge1}$ of $G$ by
$\gamma_1^p=G$ and $\gamma_n^p = [G,\gamma_{n-1}^p](\gamma_{\lceil n/p
  \rceil}^p)^p$. The corresponding abelian group
\[\LL^\Bbbk(G)=\bigoplus_{n\ge1}\gamma_n^p/\gamma_{n+1}^p\otimes_\Fp\Bbbk\]
is now a Lie algebra over $\Bbbk$, which furthermore is restricted,
with $p$-mapping defined by $(x\gamma_{n+1}^p\otimes\alpha)^p=x^p\gamma_{pn+1}^p\otimes\alpha^p$.

Note the following alternative definition~\cite{quillen:ab}: the group
ring $\Bbbk G$ is filtered by powers of its augmentation ideal
$\varpi$. The corresponding graded ring $\overline{\Bbbk
  G}=\bigoplus_{n\ge0}\varpi^n/\varpi^{n+1}$ is a Hopf algebra,
because $\varpi$ is a Hopf ideal in $\Bbbk G$. Then $\LL^\Bbbk(G)$ is the
Lie algebra of primitive elements in $\overline{\Bbbk G}$, which
itself is the universal enveloping algebra of $\LL^\Bbbk(G)$.

There is a natural graded map $\LL^\Z(G)\to\LL^\Bbbk(G)$, given by
$g\gamma_{n+1}\mapsto g\gamma_{n+1}^p\otimes1$. Furthermore, its
kernel $\KK_1$ consists of elements of the form
$(g\gamma_{n+1})^p$. There is a $p$-linear map $\KK_1\to\LL^\Bbbk(G)$,
given by $(g\gamma_{n+1})^p\mapsto g^p\gamma_{pn+1}^p\otimes1$, and
similarly for higher kernels $\KK_m$ with $m\ge2$.

We turn now to a second construction, specific for self-similar
groups. We assume, further, that $G$'s alphabet is $\Fp$, and
that the image of $\psi:G\to G\wr\Sym X$ lies in $G\wr\Fp$,
where $\Fp$ acts on itself by addition. Finally, we fix a
generating set $S$ of $G$, and assume that $\psi(S)$ lies in
$S^X\times\Fp$. Let $S'$ denote a subset of $S$ such that
every $s\in S$ is of the form $(s')^n$ for unique $s'\in S'$ and
$n\in\Fp$.

Consider now the vector space $X'=\Fp[x]/(x^p)$, and the
self-similar Lie algebra $\LL(G)$ acting faithfully on $T(X')$, with
generating set $S'$, with the following self-similarity structure: set
\[\psi'(s')=\sum_{i\in\Fp}x^i(x+1)^{p-1-i}n_i\otimes
s_i+n\partial_x\text{ for }\psi(s')=((s_i^{n_i})_{i\in\Fp},n).\]

Modify furthermore the self-similarity structure as follows, to obtain
a graded algebra in which the elements of $s'$ are homogeneous: if
$\psi'(s')=\sum_{s\in S'} f_s(x)\otimes s+n\partial_x$, then let
$\overline{f_s}$ denote the leading monomial of $f_s$, and set
\[\psi(s')=\sum_{s\in S'}\overline{f_s(x)}\otimes s+n\partial_x.\]

\begin{conj}
  Consider a self-similar group $G$ as above, its thinned algebra
  $\AA(G)$ with augmentation ideal $\varpi$, and the associated graded
  algebra
  $\overline{\AA(G)}=\bigoplus_{n\ge0}\varpi^n/\varpi^{n+1}$. Then the
  Lie algebra $\LL(G)$ is a subalgebra of $\overline{\AA(G)}$.
\end{conj}

\begin{conj}
  Consider a self-similar group $G$ as above, and $\LL^\Bbbk(G)$ its
  associated Lie algebra. Then $\LL(G)$ is a quotient of
  $\LL^\Bbbk(G)$.
\end{conj}

These conjectures should be proved roughly along the following
construction: first, there is a natural map $\pi_0:G\to
G/[G,G]\to\LL(G)$, sending $s\in S$ to its image in $\LL(G)$. Then,
given $g\gamma_{n+1}^p\in\LL(G)$, let $k\in\N$ be maximal such that
$\gamma_n^p$ acts trivially on $X^k$, and consider
$\psi^n(g)=(g_{0\dots0},\dots,g_{p-1\dots p-1})$. Write then
\[\pi_k(g)=\sum_{i_1,\dots,i_k\in\Fp} x^{i_1}(x+1)^{p-1-i_1}\otimes\cdots\otimes x^{i_k}(x+1)^{p-1-i_k}\otimes\pi_0(g_{i_1\dots i_k}).\]
These maps should induce a map $\LL^\Bbbk(G)\to\LL(G)$. A similar
construction should relate $\varpi^n\le\AA(G)$ and $\LL(G)$.

Rather than pursuing this line in its generality, wee shall now see,
in specific examples, how $\LL(G)$ and $\LL^\Bbbk(G)$ are related, and
lead from self-similar groups to the Lie algebras described
in~\S\ref{ss:examples}.

\subsection{The $p$-Sylow subgroup of the infinite symmetric group}
Kaloujnine initiated in~\cite{kaloujnine:psylow} the study of the
Sylow $p$-subgroup of $\Sym(p^m)$, and its infinite
generalization~\cite{kaloujnine:bourbaki}. The Sylow subgroup of
$\Sym(p^m)$ is an $m$-fold iterated wreath product of $C_p$, and these
groups form a natural projective system; denote their inverse limit by
$W_p$. Then $W_p\cong W_p\wr\Fp$, and if we denote this
isomorphism by $\psi$ then $W_p$ is a self-similar group satisfying
the conditions of the previous paragraph.

Sushchansky and Netreba exploited Kaloujnine's representation of
elements of $W_p$ by ``tableaux'' to describe
in~\cites{sushchansky-n:psylow,netreba-s:wreath} the Lie algebra
$\LL^\Fp(W_p)$ associated with $W_p$. This language is essentially
equivalent to ours; Kaloujnine's ``tableau'' $x_1^{e_1}\dots
x_n^{e_n}$ corresponds to our derivation
$x^{p-1-e_1}\otimes\cdots\otimes x^{p-1-e_n}\otimes\partial_x$. See
also~\cite{bartholdi:lcs}*{\S3.5} for more details. The upshot is the
\begin{prop}[\cite{bartholdi:lcs}*{Theorem~3.4}]
  The Lie algebras $\LL^\Fp(W_p)$, $\LL^\Z(W_p)$, $\LL(W_p)$ and
  $\overline{\LL(W_p)}$ are isomorphic.
\end{prop}

\subsection{The Grigorchuk groups}\label{ss:ggp}
An essential example of self-similar group was thoroughly investigated
by
Grigorchuk~\cites{grigorchuk:burnside,grigorchuk:growth,grigorchuk:(un)solved}.
The ``first'' Grigorchuk group is defined as follows: it is
self-similar; acts faithfully on $X^*$ for $X=\Ft$; is
generated by $a,b,c,d$; and has self-similarity structure
\begin{equation}\label{eq:ggp}
  \psi:\begin{cases} G\to G\wr\Ft,\\
  a\mapsto((1,1),1),\\
  b\mapsto((a,c),0),\\
  c\mapsto((a,d),0),\\
  d\mapsto((1,b),0).\end{cases}
\end{equation}

The following notable properties of $G$ stand out:
\begin{itemize}
\item it is an infinite, finitely generated torsion
  $2$-group~\cites{aleshin:burnside,grigorchuk:burnside}, providing an
  accessible answer to a question by Burnside~\cite{burnside:question}
  about the existence of such groups;
\item It has intermediate word-growth~\cite{grigorchuk:growth}, namely
  the number of group elements expressible as a word of length $\le n$
  in the generators grows faster than any polynomial, but slower than
  any exponential function. This answered a question by
  Milnor~\cite{milnor:5603} on the existence of such groups;
\item It has finite width~\cites{rozhkov:lcs,bartholdi-g:lie}, namely
  the ranks of the lower central factors $\gamma_n/\gamma_{n+1}$ is
  bounded. This disproved a conjecture by Zelmanov~\cite{zelmanov:castelvecchio}.
\end{itemize}

Grigorchuk's construction was generalized, in~\cite{grigorchuk:pgps},
to an uncountable collection of groups $G_\omega$, for
$\omega\in\{0,1,2\}^\infty=:\Omega$. They are not anymore
self-similar, but rather are related to each other by homomorphisms
$\psi:G_\omega\to G_{\sigma\omega}\wr\Ft$, where
$\sigma:\Omega\to\Omega$ is the one-sided shift. Interpret $0,1,2$
as the three non-trivial homomorphisms $\mathbb F_4\to\mathbb
F_2$. Then each $G_\omega$ is generated by $\mathbb F_4\sqcup\{a\}$,
and
\[\psi:\begin{cases} \hfill G_\omega &\to G_{\sigma\omega}\wr\Ft,\\
  \hfill a &\mapsto((1,1),1),\\
  \mathbb F_4\ni v&\mapsto((a^{\omega(v)},v),0).\end{cases}\]
The ``first'' Grigorchuk group is then $G_\omega$ for $\omega=(012)^\infty$.

The structure of the Lie algebra $\LL^\Z(G)$, based on calculations
in~\cite{rozhkov:lcs}, and of $\LL^\Ft(G)$, are described
in~\cite{bartholdi-g:lie}, and more explicitly
in~\cite{bartholdi:lcs}*{Theorem~3.5}. Note, however, some missing
arrows in~\cite{bartholdi:lcs}*{Figure~2} between $\mathbb1^n(x^2)$
and $\mathbb1^{n+2}(x)$. Notice also that $\LL^\Ft(G)$ is neither just
infinite nor centreless: its center is spanned by $\{W(x^2)\mid
W\in\{\mathbb0,\mathbb1\}^*\setminus\{\mathbb1\}^*\}$ and has finite
codimension.

Recall the \emph{upper central series} of a Lie algebra $\LL$: it is
defined inductively by $\zeta_0=0$; by
$\zeta_{n+1}/\zeta_n=\zeta(\LL/\zeta_n)$; and by
$\zeta_\omega=\bigcup_{\alpha<\omega}\zeta_\alpha$. In particular,
$\zeta_1$ is the centre of $\LL$.
\begin{thm}\label{thm:isogg}
  The Lie algebras $\LL(G)$, $\LL^\Ft(G)/\zeta(\LL^\Ft(G))$ and $\LLtg$ are isomorphic.

  The Lie algebras $\LL^\Z(G)/\zeta_\omega(\LL^\Z(G))$ and $\LLg$ are isomorphic.
\end{thm}
\begin{proof}
  We recall from~\cite{bartholdi:lcs}*{Theorem~3.5} the following
  explicit description of $\LL^\Ft(G)$. Write $e=[a,b]$. A basis of $\LL^\Ft(G)$ is
  \[\{a,b,d,[a,d]\}\cup\{W(e),W(e^2)\mid
  W\in\{\mathbb0,\mathbb1\}^*\},\] where $a,b,d$ have degree $1$;
  where $[a,d]$ has degree $2$; and where $\deg(X_1\dots
  X_n(e))=1+\sum_{i=1}^nX_i2^{i-1}+2^n=\frac12\deg(X_1\dots
  X_n(e^2))$. On that basis, the $2$-mapping sends $W(e)$ to $W(e^2)$
  and $\mathbb1^n(e^2)$ to $\mathbb1^{n+2}(e)+\mathbb1^{n+1}(e^2)$ and
  all other basis vectors to $0$.

  In particular, the centre of $\LL^\Ft(G)$ is spanned by
  $\{W(e^2)\mid W\not\in\{\mathbb1\}^*\}$, and
  $\LL^\Ft(G)/\zeta(\LL^\Ft(G))$ is centreless.

  Note then the following isomorphism between
  $\LL^\Ft(G)/\zeta(\LL^\Ft(G))$ and $\LLtg$. It sends $a,b,d,[d,a]$
  to $a,b,d,[d,a]$ respectively; and $X_1\dots X_n(e^s)$ to
  $x^{1-X_1}\otimes\cdots\otimes x^{1-X_n}\otimes e^s$.

  It follows that $\LL^\Ft(G)/\zeta(\LL^\Ft(G))$ admits the injective
  self-similarity structure~\eqref{eq:grla}, and therefore equals
  $\LLtg$. On the other hand, direct inspection shows that the
  self-similarity structure of $\LL(G)$ equals that of $\LLtg$.

  In the description of $\LL^\Z(G)$, the basis element $W(e^2)$ has
  degree $1+\sum_{i=1}^nX_i2^{i-1}+2^{n+1}$. The centre of $\LL^\Z(G)$
  is spanned by the $\{\mathbb1^n(e^2)\}$, and more generally
  $\zeta_k(\LL^\Z(G))$ is spanned by the $\{X_1\dots X_n(e^2)\mid
  \sum_{i=1}^nX_i2^{i-1}\ge 2^n-k$. It then follows that
  $\LL^\Z(G)/\zeta_\omega(\LL^\Z(G))$ has basis
  $\{a,b,d,[a,d]\}\cup\{W(e)\}$, and admits the injective
  self-similarity structure~\eqref{eq:grla}, so it equals $\LLg$.
\end{proof}

\begin{thm}\label{thm:isothinnedg}
  The thinned algebras associated with $G$ and $\LL(G)$ are
  isomorphic: $\AA(G)\cong\AA(\LL(G))$.
\end{thm}
\begin{proof}
  The algebra $\AA(G)$ is just
  infinite~\cite{bartholdi:branchalgebras}*{Theorem~4.3}, and has an
  explicit presentation
  \[\AA(G)=\langle A,B,C,D\mid\mathcal
  R_0,\sigma^n(CACACAC),\sigma^n(DACACAD)\text{ for all }n\ge0\rangle,\]
  where $\sigma:\{A,B,C,D\}^*\to\{A,B,C,D\}^*$ is the substitution
  \[A\mapsto ACA,\quad B\mapsto D,\quad C\mapsto B,\quad D\mapsto C\]
  and
  \[\mathcal R_0=\{A^2,B^2,C^2,D^2,B+C+D,BC,CB,BD,DB,CD,DC,DAD\}.\]
  It is easy to check that $\psi(\sigma(w))=(\begin{smallmatrix} 0&0\\
    w&0\end{smallmatrix})$ for all words $w\in\{A,B,C,D\}$ of length
  at least $2$ and starting and ending in $\{B,C,D\}$; and that the
  relations $\mathcal R_0\cup\{CACACAC,DACACAD\}$ are satisfied in
  $\AA(\LL(G))$. There exists therefore a homomorphism
  $\AA(G)\to\AA(\LL(G))$, which must be an isomorphism because its
  image has infinite dimension.
\end{proof}

\subsection{The Gupta-Sidki group}\label{ss:gsgp}
Another important example of self-similar group was studied by Gupta
and Sidki~\cite{gupta-s:burnside}. This group is defined as follows:
it is self-similar; acts faithfully on $X^*$ for $X=\mathbb F_3$; is
generated by $a,t$; and has self-similarity structure
\begin{equation}\label{eq:gsgp}
  \psi:\begin{cases} G\to G\wr\mathbb F_3,\\
    a\mapsto((1,1,1),1),\\
    t\mapsto((a,a^{-1},t),0).\end{cases}
\end{equation}
Gupta and Sidki prove that $G$ is an infinite, finitely generated
torsion $3$-group.

The Lie algebra $\LL^\Z(G)=\LL^{\mathbb F_3}(G)$ is described
in~\cite{bartholdi:lcs}*{Theorem~3.8}, where is it shown that
$\LL^\Z(G)$ has unbounded width.

\begin{thm}\label{thm:isogs}
  The Lie algebras $\LL(G)$, $\LL^{\mathbb F_3}(G)$, $\LL^\Z(G)$ and $\LLgs$ are
  isomorphic.
\end{thm}
\begin{proof}
  We recall from~\cite{bartholdi:lcs}*{Theorem~3.8} the following
  explicit description of $\LL^\Fp(G)$. Write $c=[a,t]$ and
  $u=[a,c]$. Define the integers $\alpha_n$ by $\alpha_1=1$,
  $\alpha_2=2$, and $\alpha_n=2\alpha_{n-1}+\alpha_{n-2}$ for $n\ge3$.
  A basis of $\LL^\Fp(G)$ is
  \[\{a,t\}\cup\{W(c),W(u)\mid
  W\in\{\mathbb0,\mathbb1,\mathbb2\}^*\},\] where $a,t$ have degree
  $1$, and where $\deg(X_1\dots
  X_n(c))=1+\sum_{i=1}^nX_i\alpha_i+\alpha_{n+1}=\deg(X_1\dots
  X_n(u))-\alpha_{n+1}$. On that basis, the $3$-mapping sends
  $\mathbb2^n(c)$ to $\mathbb2^n\mathbb0^2(c)+\mathbb2^n\mathbb1(u)$
  and all other basis vectors to $0$.

  In particular, $\LL^\Z(G)$ and $\LL^{\mathbb F_3}(G)$ are isomorphic.

  Note then the following isomorphism between $\LL^Z(G)$ and
  $\LLtg$. It sends $a,t$ to $a,t$ respectively; and $X_1\dots X_n(b)$
  to $x^{2-X_1}\otimes\cdots\otimes x^{2-X_n}\otimes b$ for
  $b\in\{c,u\}$.

  It follows that $\LL^\Z(G)$ admits the injective self-similarity
  structure~\eqref{eq:gsla}, and therefore equals $\LLgs$. On the
  other hand, direct inspection shows that the self-similarity
  structure of $\LL(G)$ equals that of $\LLgs$.
\end{proof}

It is tempting to conjecture, in view of
Theorem~\ref{thm:isothinnedg}, that the associated graded
$\bigoplus_{n\ge0}\varpi^n/\varpi^{n+1}$ of $\AA(G)$ is isomorphic to
$\AA(\LL)$; presumably, this could be proven by finding a presentation
of $\AA(G)$.

\subsection{From Lie algebras to groups}
We end with some purely speculative remarks. Although we gave a
construction of a Lie algebra starting from a self-similar group, this
construction depends on several choices, in particular of a generating
set for the group. Could it be that the resulting algebra $\LL(G)$, or
$\overline{\LL(G)}$, is independent of such choices? Is
$\overline{\LL(G)}$ always isomorphic to $\LL^p(G)$?

On the other hand, I do not know of any ``interesting'' group to
associate with a self-similar Lie algebra --- that would, for example,
be a torsion group if the Lie algebra is nil, or have subexponential
growth if the Lie algebra has subexponential growth.

A naive attempt is the following. Consider the ``Fibonacci'' Lie
algebra $\LL_{2,\Ft}$ from~\S\ref{ss:pszla}. Its corresponding
self-similar group should have alphabet $X=\Ft$, and
generators $a,t$ with self-similarity structure
\[\psi:\begin{cases} G\to G\wr\Ft\\
  a\mapsto((1,1),1),\\
  t\mapsto((a,at),0), \end{cases}\] at least up to commutators. That
group can easily be shown to be contracting, and also regularly weakly
branched on the subgroup $\langle[a,t,t]\rangle^G$. It has Hausdorff
dimension $1/3$, and probably exponential growth.

More generally, are there subgroups of the units in $\AA(\LL)$ that
are worth investigation, for $\LL$ one of the Lie algebras
from~\S\ref{ss:examples}?

\end{document}